\documentclass[a4paper]{amsart}

\usepackage[hmarginratio={1:1},vmarginratio={1:1},height=584pt,width=370pt,tmargin=117pt]{geometry}

\usepackage{amsmath,amssymb,amscd,latexsym,epic,bbm,comment,mathbbol,amsthm}
\usepackage{graphicx,enumerate,stmaryrd,color,xcolor}

\usepackage[all,2cell]{xy}
\xyoption{2cell}

\usepackage[active]{srcltx}

\usepackage{etoolbox}
\makeatletter
\let\ams@starttoc\@starttoc
\makeatother
\usepackage[parfill]{parskip}
\makeatletter
\let\@starttoc\ams@starttoc
\patchcmd{\@starttoc}{\makeatletter}{\makeatletter\parskip\z@}{}{}
\makeatother

\newtheorem{theorem}{Theorem}[section]
\newtheorem{lemma}[theorem]{Lemma}
\newtheorem{corollary}[theorem]{Corollary}
\newtheorem{proposition}[theorem]{Proposition}

\theoremstyle{definition}

\newtheorem{example}[theorem]{Example}
\newtheorem{remark}[theorem]{Remark}

\definecolor{orchid}{RGB}{143,40,194}
\definecolor{lava}{RGB}{207,16,32}

\usepackage{hyperref}

\hypersetup{
    pdftoolbar=true,        
    pdfmenubar=true,        
    pdffitwindow=false,     
    pdfstartview={FitH},    
    pdftitle={\texorpdfstring{$2$}{2}-representations via (co)algebra \texorpdfstring{$1$}{1}-morphisms},    
    pdfauthor={Marco Mackaay, Volodymyr Mazorchuk, Vanessa Miemietz and Daniel Tubbenhauer},     
    pdfsubject={},   
    pdfcreator={Marco Mackaay, Volodymyr Mazorchuk, Vanessa Miemietz and Daniel Tubbenhauer},   
    pdfproducer={Marco Mackaay, Volodymyr Mazorchuk, Vanessa Miemietz and Daniel Tubbenhauer}, 
    pdfkeywords={}, 
    pdfnewwindow=true,      
    colorlinks=true,       
    linkcolor=lava,          
    citecolor=teal,        
    filecolor=magenta,      
    urlcolor=orchid,          
    linkbordercolor=red,
    citebordercolor=teal,
    urlbordercolor=orchid,  
    linktocpage=true
}

\newcommand{\somespace}{\scalebox{1.75}{\raisebox{-.075cm}{$\phantom{I}$}}\!\!\!\!}

\font\sc=rsfs10
\newcommand{\cC}{\sc\mbox{C}\hspace{1.0pt}}
\newcommand{\cI}{\sc\mbox{I}\hspace{1.0pt}}
\newcommand{\cS}{\sc\mbox{S}\hspace{1.0pt}}
\newcommand{\cT}{\sc\mbox{T}\hspace{1.0pt}}
\newcommand{\cL}{\sc\mbox{L}\hspace{1.0pt}}
\newcommand{\cA}{\sc\mbox{A}\hspace{1.0pt}}
\newcommand{\cQ}{\sc\mbox{Q}\hspace{1.0pt}}

\font\scc=rsfs7
\newcommand{\ccC}{\scc\mbox{C}\hspace{1.0pt}}
\newcommand{\ccA}{\scc\mbox{A}\hspace{1.0pt}}
\newcommand{\ccS}{\scc\mbox{S}\hspace{1.0pt}}

\begin{document}

\title[$2$-representations via (co)algebra $1$-morphisms]{Simple transitive $2$-representations\\
via (co)algebra $1$-morphisms}

\author[M. Mackaay, V. Mazorchuk, V. Miemietz and D. Tubbenhauer]{Marco Mackaay, 
Volodymyr Mazorchuk,\\ Vanessa Miemietz and Daniel Tubbenhauer}

\vbadness=10001
\hbadness=10001
\begin{abstract}
For any fiat $2$-category $\cC$, 
we show how its simple transitive 
$2$-re\-presentations can be constructed using coalgebra $1$-morphisms in the 
injective abelianization of $\cC$.
Dually, we show that these can 
also be constructed 
using
algebra $1$-morphisms in the projective abelianization of $\cC$. 
We also extend Mo\-rita--Takeuchi theory 
to our setup and work out several 
examples, including that of Soergel bimodules for dihedral groups, explicitly.
\end{abstract}

\maketitle
\tableofcontents
\section{Introduction}\label{s0}

The subject of $2$-representation theory, which has its origins in \cite{CR,KhLa,Ro}, 
is the higher categorical 
analogue of the classical representation theory of algebras. 
A systematic study of 
the ``finite dimensional'' counterpart of $2$-representation 
theory started in \cite{MM1}, and was 
continued in \cite{MM2}--\cite{MM6}, see 
also \cite{Xa,Zh1,Zh2,MZ} and the references therein. 
In particular, for a given finitary 
$2$-category $\cC$, the paper \cite{MM5} defines an
appropriate $2$-analog of simple representations, 
the so-called {\em simple transitive}
$2$-re\-presentations, and uses these $2$-representations 
to establish a 
Jordan--H\"{o}lder theory for finitary $2$-categories. 
This motivates the problem of classifying the
simple transitive $2$-re\-presentations of a given finitary 
$2$-category $\cC$. For instance, this question was
studied, for various $2$-categories, 
in e.g. \cite{MM5,MM6,Zh2,Zi,MZ,MaMa,KMMZ,MT}.

An important example of such simple transitive $2$-representations 
is provided by the so-called
\textit{cell $2$-representations} defined in \cite{MM1,MM2} using 
combinatorics of 
$1$-mor\-phisms in $\cC$. The notion of cell $2$-representation is 
inspired by the 
Kazhdan--Lusztig cell representations of Hecke algebras of Coxeter 
groups \cite{KL0}. 
In some cases, for example for the $2$-categories of 
Soergel bimodules 
in type $A$, cell $2$-re\-presentations exhaust simple transitive 
$2$-representations, as was shown in \cite{MM5}. However, it turns 
out that, in many cases, there 
are simple transitive $2$-representations which are not equivalent 
to cell $2$-representations.
The first, very degenerate, examples already appeared in \cite{MM5}. 
However, the first interesting, 
and unexpected, example appeared in \cite{MaMa} 
which studies simple
transitive $2$-representations for some subquotients of Soergel 
bimodules in dihedral Coxeter types
$I_2(4)$ and $I_2(5)$.

The subsequent paper \cite{KMMZ} studies 
the classification of 
simple transitive $2$-re\-pre\-sen\-tations for so-called \textit{small 
quotients of Soergel bimodules} in all 
finite Weyl types. In particular, the existence of 
simple transitive 
$2$-representations which are not cell $2$-representations was 
established for 
all dihedral Coxeter types $I_2(2n)$, where $n>2$. The classification 
problem was completed for
all types with the exception of $I_2(12)$, $I_2(18)$ and 
$I_2(30)$. 

The classification of simple transitive $2$-representations 
is usually approached in two steps. The first step
addresses the classification of certain integral representations of the group algebra of
the corresponding Weyl group. As it turns out 
in \cite{KMMZ}, for dihedral types, this latter classification
is given in terms of simply laced Dynkin diagrams. 
The type $A$ Dynkin diagrams lead to 
cell $2$-representations, and the type $D$ Dynkin diagrams lead
to new simple transitive 
$2$-representations which, together with the cell $2$-representations, 
exhaust all simple
transitive $2$-representations unless the dihedral group is of type $I_2(12)$, $I_2(18)$ or $I_2(30)$. 
(These three cases correspond to type $E$ Dynkin diagrams, with 
$12$, $18$ and $30$ being the Coxeter numbers of $E_6$, $E_7$ and $E_8$).

The problem with these three exceptional types 
was that the classification of integral
representations of the group algebra of the corresponding 
Weyl group predicted the 
existence of additional ``type $E$'' simple transitive 
$2$-representations.
These additional simple transitive $2$-representations were 
constructed later in \cite{MT} (relying on ideas from \cite{KS,AT}),
using a presentation for Soergel bimodules 
given in \cite{El}. This method differs conceptually
from the one used in \cite{MM5,KMMZ} (and the
other papers mentioned above)
where $2$-representations were constructed as 
``subquotients'' of the so-called \textit{principal $2$-representations}.
At the moment, there is no universal algorithm which would allow one to move between 
the $2$-representations constructed using these two different methods.

The main motivation for the present paper 
is to develop some techniques 
to reinterpret the results of \cite{MT} in 
the framework of the approach of \cite{MM4,KMMZ}. For this 
we extend to our setup 
and further develop the ideas 
of \cite{Os,EO,ENO,EGNO} which study $2$-representations 
of certain tensor categories
using algebra objects in these tensor categories and 
the module
categories associated to these algebra objects.

An interesting example 
is given by the semisimplified quotient of 
$\mathrm{U}_q(\mathfrak{sl}_2)$-$\mathrm{mod}$, 
where $q$ is a primitive complex even root of unity. 
In this case,
as it is shown 
in \cite{KO,Os}, there 
are three families of algebra objects. 
These correspond to the simply laced Dynkin diagrams, 
just like the simple transitive $2$-representations 
of Soergel bimodules in dihedral types.

In the present paper, we show that, for a 
given fiat $2$-category $\cC$ (in the sense of \cite{MM1}), 
there is a bijection between the equivalence classes 
of simple transitive $2$-representations of $\cC$ and the Morita--Takeuchi
equivalence classes of simple coalgebra $1$-morphisms in $\underline{\cC}$, 
the injective abelianization of $\cC$, see Theorem \ref{thm7}.
Dually, we also show that there is a bijection between the equivalence classes 
of simple transitive $2$-representations of $\cC$ and the 
Morita equivalence classes of 
algebra $1$-morphisms in $\overline{\cC}$, 
the projective abelianization of $\cC$, see Corollary \ref{cor-72}. 
Here the Morita(--Takeuchi) theory for (co)algebra 
$1$-morphisms in $\cC$ 
is a direct generalization of the classical Morita(--Takeuchi) theory for 
(co)algebras \cite{Mo,Ta}, as we explain in Section \ref{section:Morita-Takeuchi}.
(This is not to be confused with the results of \cite{MM4}.)

Our results extend and 
generalize some of the results in \cite{Os,EO,ENO,EGNO}.
However, there are some essential 
difficulties due to the fact that our setup differs
from the one studied in \cite{Os,EO,ENO,EGNO}: 
For example, the latter references work mostly with  abelian
monoidal categories (with some extra structure), while the categories 
we consider are additive, but almost never abelian.
One of the manifestations of this difficulty 
is our definition of {\em internal homs} in 
Section \ref{s5.1} which is, in some sense, 
dual to the one used in \cite{EGNO}. 
Indeed, it turns out that the internal hom 
defined in \cite[Definition 7.9.2]{EGNO}  
does not have the necessary properties which would allow 
one to develop a useful theory in our setup. See also 
Remark \ref{remark:egno-vs-our-stuff}.

We also give several 
(classes of) examples. One of 
them explains the relation between 
the two aforementioned ADE classifications. 
For this we crucially rely on \cite{El}. That is, 
Elias' results show that there 
is a $2$-functor between 
the semisimplified quotient of $\mathrm{U}_q(\mathfrak{sl}_2)$-$\mathrm{mod}$, 
with $q$ as above, and the small quotient of the $2$-category of singular Soergel 
bimodules of dihedral type. We show that this $2$-functor gives the link between the 
two ADE classifications, using 
our relation between simple transitive $2$-representations 
and algebra $1$-morphisms, see Section \ref{s8}. 

\begin{remark}\label{remark:upshots}
It is worth emphasizing that the 
approach using (co)algebra $1$-mor\-phism does not 
seem to be very
helpful for the classification of $2$-representations, 
because the  classification of (co)algebra $1$-morphisms looks
like a very hard problem in general. However, this method is quite 
helpful if one would like to check existence
of some $2$-re\-presentations, since this can be 
reformulated into the problem of 
checking that certain $1$-mor\-phisms have an additional 
structure of a (co)algebra $1$-mor\-phism. 
This is sometimes quite easy, e.g. the type A and D algebra $1$-morphisms for 
the semisimplified quotient of $\mathrm{U}_q(\mathfrak{sl}_2)$-$\mathrm{mod}$ 
decategorify to idempotents in the Grothendieck group -- and 
this basically fixes the algebra structure, cf. Remark \ref{remark:typeADE-2}.
\end{remark}

The paper is organized as follows: Section \ref{s1} 
contains preliminaries on $2$-rep\-re\-sen\-ta\-ti\-ons.
Section \ref{s4} introduces a new version of 
abelianization for finitary $2$-categories and
compares it to the previous versions defined in \cite{MM1}. 
This new abelianization is essential in 
the rest of the paper as it significantly simplifies 
arguments related to abelianization of $2$-representations.
Section \ref{s5} contains our main results
mentioned above. Section \ref{section:Morita-Takeuchi}
establishes an analogue of Morita(--Takeuchi) theory in our 
setup. Finally, Sections \ref{s7} and \ref{s8}
deal with some explicit examples and applications, which 
include $2$-categories of projective
functors for finite dimensional algebras and 
$2$-categories of Soergel bimodules 
of dihedral type.
\vspace{0.5cm}
 
\textbf{Acknowledgements:} The second 
author is partially supported by
the Swedish Research Council, Knut and 
Alice Wallenberg Stiftelse and G{\"o}ran Gustafsson Stiftelse.
A part of this paper was written during the visit of 
the third author to Uppsala University
in September 2016. This visit was supported by the 
Swedish Research Council. Both, this support and
the hospitality of Uppsala University, are gratefully acknowledged.
The fourth author thanks the 
Hausdorff Center for Mathematics (HCM) in Bonn
for partially sponsoring a research visit during this project, 
as well as Darkness, his old friend.
We would also like to thank Ben Elias for stimulating discussions.
We thank the referee for an extremely careful reading of the manuscript, 
for many very helpful comments.

\section{Some recollections of \texorpdfstring{$2$}{2}-representation theory}\label{s1}

\subsection{Basic notation and conventions}\label{s1.1}

We fix an algebraically closed field $\mathbbm{k}$. All of our 
($2$-)categories and ($2$-)functors etc.\
will be $\mathbbm{k}$-linear unless stated otherwise.

A {\em $2$-category} is a category enriched over 
the category of all (small) categories. 
That is, a $2$-category $\cC$ consists of a collection 
of objects denoted by $\mathtt{i},\mathtt{j},\mathtt{k}$ etc.;
for each pair $(\mathtt{i},\mathtt{j})$ of objects, a 
small category $\cC(\mathtt{i},\mathtt{j})$ consisting of a 
set of $1$-morphisms,
whose elements will be denoted by 
$\mathrm{F},\mathrm{G},\mathrm{H}$ etc., and, for each pair $(\mathrm{F},\mathrm{G})$ of 
$1$-morphisms in a fixed $\cC(\mathtt{i},\mathtt{j})$, a set 
$\mathrm{Hom}_{\ccC(\mathtt{i},\mathtt{j})}(\mathrm{F},\mathrm{G})$ 
of $2$-morphisms 
(we also write $\mathrm{Hom}_{\ccC}(\mathrm{F},\mathrm{G})$ etc.\
for short), whose elements we call 
$\alpha,\beta,\gamma$ etc. For any $\mathtt{i}\in \cC$, 
we denote by $\mathbbm{1}_{\mathtt{i}}$ the identity 
$1$-mor\-phism in $\cC(\mathtt{i},\mathtt{i})$.
For a $1$-morphism $\mathrm{F}$, we 
write $\mathrm{id}_{\mathrm{F}}$ for the corresponding 
identity $2$-mor\-phism
in $\mathrm{Hom}_{\ccC(\mathtt{i},\mathtt{j})}(\mathrm{F},\mathrm{F})$.
Moreover, we use the symbol $\circ$ for composition of $1$-morphisms (but often omit it), 
$\circ_0$ for horizontal composition
and $\circ_1$ for vertical composition of $2$-morphisms.
(For more background on abstract $2$-categories, see e.g. \cite{ML,Le}.)

The $2$-category $\cC$ is called \textit{$\mathbbm{k}$-linear} 
if $\mathrm{Hom}_{\ccC(\mathtt{i},\mathtt{j})}(\mathrm{F},\mathrm{G})$ is a 
$\mathbbm{k}$-vector space, for all $(\mathrm{F},\mathrm{G})$, and if 
horizontal and vertical compositions are $\mathbbm{k}$-bilinear operations. 

Moreover, we will also meet the notion of a \textit{bicategory}. 
We do not need bicategories often in this paper and refer the reader 
to \cite{Be,ML,Le} for details. The example to keep in mind are 
(non-strict) monoidal categories.
The most important, for us,  fact to recall about bicategories 
is that they can always be \textit{strictified}: any bicategory is weakly 
equivalent to a $2$-category, 
see e.g. \cite{Be} or \cite[Theorem 2.3]{Le}.

\subsection{Finitary categories and \texorpdfstring{$2$}{2}-categories}\label{s1.2}

An additive $\mathbbm{k}$-linear category 
is called {\em finitary} if it has split idempotents, only finitely 
many isomorphism classes of indecomposable 
objects and the morphism sets are finite dimensional 
$\mathbbm{k}$-vec\-tor spaces. We write 
$\mathfrak{A}^f_{\mathbbm{k}}$ for the $2$-category which has
\begin{itemize}
\item finitary additive $\mathbbm{k}$-linear categories as objects;
\item  $\mathbbm{k}$-linear (hence, additive) functors as $1$-morphisms;
\item natural transformations of functors as $2$-morphisms.
\end{itemize}

Then we say a $2$-category $\cC$ is {\em finitary} provided
\begin{itemize}
\item it only has finitely many objects;
\item for each pair $(\mathtt{i},\mathtt{j})$ of objects, 
the category $\cC(\mathtt{i},\mathtt{j})$ is in $\mathfrak{A}_{\mathbbm{k}}^f$;
\item horizontal composition is additive and $\mathbbm{k}$-linear;
\item the identity $1$-morphism 
$\mathbbm{1}_{\mathtt{i}}$ is indecomposable for every $\mathtt{i}\in\cC$.
\end{itemize}

\subsection{Fiat \texorpdfstring{$2$}{2}-categories}\label{s1.5}

For any $2$-category $\cC$, we consider the $2$-category $\cC^{\,\mathrm{co,op}}$, which is obtained from 
$\cC$ by reversing both $1$- and $2$-morphisms. 

We say that a finitary $2$-category $\cC$ is {\em weakly fiat} if 
$\cC$ is endowed with a weak equivalence ${}^{*}\colon\cC\to \cC^{\,\mathrm{co,op}}$ such that,
for any pair $(\mathtt{i}, \mathtt{j})$ of objects and every $1$-mor\-phism
$\mathrm{F}\in\cC(\mathtt{i},\mathtt{j})$, there are
$2$-mor\-phisms $\alpha\colon\mathrm{F}\circ\mathrm{F}^*\to
\mathbbm{1}_{\mathtt{j}}$ and $\beta\colon\mathbbm{1}_{\mathtt{i}}\to
\mathrm{F}^*\circ\mathrm{F}$ satisfying
$\alpha_{\mathrm{F}}\circ_1\mathrm{F}(\beta)=\mathrm{id}_{\mathrm{F}}$ and
$\mathrm{F}^*(\alpha)\circ_1\beta_{\mathrm{F}^*}=\mathrm{id}_{\mathrm{F}^*}$. Note that $\alpha$ 
and $\beta$ define an adjunction between $F$ and $F^\ast$. We therefore call them 
{\em adjunction $2$-morphisms}. 

We denote by ${}^*\mathrm{F}$ the image of a $1$-morphism $\mathrm{F}$ under an inverse to ${}^*$.

If ${}^{*}$ is a weak involution, we say that $\cC$ is {\em fiat}. 
(We refer the reader to 
\cite{MM1,MM2,MM6} for details about (weakly) fiat categories.)

\subsection{\texorpdfstring{$2$}{2}-representations}\label{s1.3}

Let $\cC$ be a finitary $2$-category. By a {\em $2$-representation} 
of $\cC$ we mean a strict 
$2$-functor from $\cC$ to the $2$-category of (small) categories. 

Such a 
$2$-representation is called a {\em finitary $2$-representation} if it is 
a strict $2$-functor from $\cC$ to $\mathfrak{A}_{\mathbbm{k}}^f$. 
We usually denote $2$-representations 
by $\mathbf{M},\mathbf{N},\dots$. For any fixed $\mathtt{i}\in\cC$, we 
use the symbol $\mathbf{P}_{\mathtt{i}}$ for the
$\mathtt{i}$-th {\em principal} $2$-representation $\cC(\mathtt{i},{}_-)$. 

All finitary $2$-representations of $\cC$ form a $2$-category
whose $1$-morphisms are $2$-natural transformations and whose 
$2$-morphisms are modifications 
(more details are given in \cite{Le,MM3}).
Moreover, we say that two $2$-representations 
$\mathbf{M}$ and $\mathbf{N}$ of $\cC$ are {\em equivalent}, if there exists 
a $2$-natural transformation $\Phi\colon\mathbf{M}\to\mathbf{N}$ which induces 
an equivalence of categories, for each object $\mathtt{i}$.

\subsection{\texorpdfstring{$2$}{2}-ideals}\label{s1.6}

A semicategory is a collection of objects and morphisms satisfying the axioms 
of a category except for the existence of identity morphisms. Similarly, a 
$2$-semicategory is a category enriched over semicategories.

Given any $2$-category $\cC$, a {\em left $2$-ideal} $\cI$ of $\cC$ is a $2$-semicategory, 
which has the same objects as $\cC$ and in which, for each pair $(\mathtt{i},\mathtt{j})$ 
of objects, $\cI(\mathtt{i},\mathtt{j})$ is an ideal in  $\cC(\mathtt{i},\mathtt{j})$, 
closed under left horizontal multiplication with both $1$- and $2$-morphisms in $\cC$. 
Similarly, one defines {\em right $2$-ideals} and {\em two-sided $2$-ideals} 
(which we also just call {\em $2$-ideals}).
An important class of left $2$-ideals in $\cC$ is given by 
the $\mathtt{i}$-th principal $2$-re\-presentations $\mathbf{P}_{\mathtt{i}}$.

If $\mathbf{M}$ is a $2$-representation 
$\cC$, an {\em ideal}
$\mathbf{I}$ in $\mathbf{M}$ is the data of 
an ideal $\mathbf{I}(\mathtt{i})$ in 
$\mathbf{M}(\mathtt{i})$, for each $\mathtt{i}\in\cC$, which is stable under the action of $\cC$.

\subsection{Simple transitive \texorpdfstring{$2$}{2}-representations}\label{s2.4}

Let $\cC$ be a finitary $2$-category. We call a 
finitary $2$-representation $\mathbf{M}$ of $\cC$ 
{\em transitive} if, for every object $\mathtt{i}\in\cC$ 
and every non-zero object $X\in\mathbf{M}(\mathtt{i})$, 
the $2$-sub\-representation $\mathbf{G}_{\mathbf{M}}(X)$ of 
$\mathbf{M}$ is equivalent to $\mathbf{M}$. Here $\mathbf{G}_{\mathbf{M}}(X)$ 
is the additive closure $\mathrm{add}(\{\mathbf{M}(\mathrm{F})X\})$, where $\mathrm{F}$ 
runs over all $1$-morphisms of $\cC$. In what follows we will often
use the module (action) notation $\mathrm{F}\, X$ instead of the 
representation notation $\mathbf{M}(\mathrm{F})X$. 

A  transitive $2$-representation $\mathbf{M}$ has a 
unique maximal ideal $\mathbf{I}$ not containing
 any identity morphisms other than that of the zero 
 object (cf. \cite[Lemma 4]{MM5}). If $\mathbf{I}=0$,
then $\mathbf{M}$ is said to be {\em simple transitive}. 
In general,  the quotient  $\widehat{\mathbf{M}}$ 
of $\mathbf{M}$ by $\mathbf{I}$ is simple transitive and 
is called the {\em simple transitive quotient} of $\mathbf{M}$. 

%
%

\subsection{Combinatorics of \texorpdfstring{$1$}{1}-morphisms}\label{s1.4}

Recall that a multisemigroup is a 
pair consisting of a set $\mathcal{S}$ and an associative 
multivalued operation from $\mathcal{S}\times \mathcal{S}$ to the set of subsets of $\mathcal{S}$.

Let $\cC$ be a finitary $2$-category 
and denote by $\mathcal{S}(\cC)$ the set of isomorphism classes of 
indecomposable $1$-morphisms 
in $\cC$. This has the structure of a 
multisemigroup by \cite[Section 3]{MM2}, which comes equipped with 
several preorders. 

For two $1$-morphisms $\mathrm{F}$ and $\mathrm{G}$, we have
$\mathrm{G}\geq_L\mathrm{F}$ in the {\em left preorder} 
if there exists a $1$-morphism $\mathrm{H}$ such that 
$\mathrm{G}$ occurs, up to isomorphism, as a direct 
summand in $\mathrm{H}\circ \mathrm{F}$. 
An equivalence class for this preorder is called a {\em left cell}. 
Similarly, we define the {\em right} and 
{\em two-sided} preorders $\geq_R$ and $\geq_J$, and the 
corresponding {\em right} and {\em two-sided} 
cells.

Observe that $\geq_L$ defines a partial order on 
the set of left cells, and, similarly, 
$\geq_R$ and  $\geq_J$ define partial orders on the 
sets of right cells and two-sided cells.

Note that, if $\cC$ is weakly fiat, then both
$\mathrm{F}\mapsto \mathrm{F}^*$ and $\mathrm{F}\mapsto {}^*\mathrm{F}$ induce isomorphisms
between the partially ordered sets $\left(\mathcal{S}(\cC),\leq_L\right)$ and $\left(\mathcal{S}(\cC),\leq_R\right)$.

\subsection{Cell \texorpdfstring{$2$}{2}-representations}\label{s2.1}

Let $\cC$ be a finitary $2$-category.
Consider a $2$-re\-pre\-sen\-ta\-tion $\mathbf{M}$ of $\cC$ such that, 
for each object $\mathtt{i}\in\cC$, the category $\mathbf{M}(\mathtt{i})$ is additive and
idempotent complete. Let $I$ be a subset of 
objects in $\cC$. Given any collection $X_i\in \mathbf{M}(\mathtt{i}_i)$ of 
objects, where $i\in I$, define $\mathbf{G}_{\mathbf{M}}(\{X_i\mid i\in I\})$ as 
in Section \ref{s2.4}. This becomes a $2$-subrepresentation of $\mathbf{M}$ by restriction.

For any left cell $\mathcal{L}$ in $\cC$, there exists an object 
$\mathtt{i}=\mathtt{i}_{\mathcal{L}}\in\cC$ such that the domain of every $1$-morphism in $\mathcal{L}$
is $\mathtt{i}$. Therefore it makes sense to define the $2$-representation 
$\mathbf{N}=\mathbf{G}_{\mathbf{P}_{\mathtt{i}}}(\mathcal{L})$, 
which, by \cite[Lemma 3]{MM5}, has  a unique maximal ideal $\mathbf{I}$  not containing 
$\mathrm{id}_{\mathrm{F}}$, for any $\mathrm{F}\in\mathcal{L}$. 
We call the quotient $\mathbf{C}_{\mathcal{L}}=\mathbf{N}/\mathbf{I}$  
the {\em (additive) cell $2$-representation} of $\cC$ associated to $\mathcal{L}$.

\section{Several versions of abelianization}\label{s4}

\subsection{Classical abelianization}\label{s4.1}

Let $\mathcal{A}$ be a finitary category. Recall, see e.g. \cite{Fr}, that the (diagrammatic) {\em injective abelianization}  
$\underline{\underline{\mathcal{A}}}$ of $\mathcal{A}$ is defined as follows:
\begin{itemize}
\item objects of $\underline{\underline{\mathcal{A}}}$ are diagrams $X\overset{f}{\longrightarrow}Y$
over $\mathcal{A}$;
\item morphisms in $\underline{\underline{\mathcal{A}}}$ are equivalence classes of solid
commutative diagrams of the form (i.e. without the diagonal dashed arrow)
\[
\xymatrix{
X\ar[rr]^{f}\ar[d]_{g}&&Y\ar[d]^{h}\ar@{-->}[dll]_{q}\\
X'\ar[rr]_{f'}&&Y'
} 
\]
modulo the ideal generated by those diagrams for which there is a ``homotopy''
$q$ as shown by the dashed arrow  such that $g=qf$;
\item identity morphisms are given by diagrams in which both $g$ and $h$ are the identities;
\item composition is given by the vertical composition of diagrams.
\end{itemize}

The category $\underline{\underline{\mathcal{A}}}$ is abelian. In particular, 
the object $X\overset{f}{\longrightarrow}Y$ has an embedding into 
$X\overset{0}{\longrightarrow}0$, its injective hull, and this embedding is a kernel of $f$. 
Moreover, the category $\underline{\underline{\mathcal{A}}}$ is equivalent 
(here and further in similar situations: as a $\mathbbm{k}$-linear category) to the 
category of {\em right} finite dimensional $\mathcal{A}$-mo\-dules. The original 
category $\mathcal{A}$ embeds into $\underline{\underline{\mathcal{A}}}$ via
\[
X\mapsto \xymatrix{X\ar[r]&0} \qquad\text{ and }\qquad
f\colon X\to Y \mapsto\raisebox{.75cm}{
\xymatrix{
X\ar[r]\ar[d]_{f}&0\ar[d]\\
Y\ar[r]&0
}} 
\]
and this embedding induces an equivalence between $\mathcal{A}$ and the full subcategory of 
injective objects in $\underline{\underline{\mathcal{A}}}$.

The (diagrammatic) {\em projective abelianization} $\overline{\overline{\mathcal{A}}}$
is defined in the dual way, see e.g. \cite{Fr} or \cite[Section 3.1]{MM1}.

\subsection{A different version of injective abelianization}\label{s4.2}

Let $\mathcal{A}$ be a finitary category.
Now we define a slightly different version $\underline{\mathcal{A}}$ of  
the 
(diagrammatic) {\em injective abelianization} of $\mathcal{A}$ in the following way:
\begin{itemize}

\item objects in $\underline{\mathcal{A}}$ are
tuples of the form $(X,k,Y_i,f_i)_{i=1}^{\infty}$, where 
$k\in\mathbb{Z}_{\geq 0}$, $X$ and $Y_i$ are objects in $\mathcal{A}$,
and $f_i\colon X\to Y_i$ are morphisms in $\mathcal{A}$, with the additional
requirement that  $Y_i=0$ for all $i>k$;

\item morphisms in $\underline{\mathcal{A}}$ 
from $(X,k,Y_i,f_i)_{i=1}^{\infty}$ to $(X',k',Y'_i,f'_i)_{i=1}^{\infty}$
are equivalence classes of tuples $(g,h_{i,j})_{i,j=1}^{\infty}$, where 
$g\colon X\to X'$ and $h_{i,j}\colon Y_i\to Y'_j$ are morphisms in $\mathcal{A}$
such that $f'_ig=\sum_j h_{j,i}f_j$, for each $i$, modulo 
the equivalence relation given by the homotopy relation spanned by those 
tuples $(g,h_{i,j})$ for which there exist 
$\displaystyle q_i\colon Y_i\to X'$ such that $\sum_i q_if_i=g$;

\item identity morphisms are given by tuples $(g,h_{i,j})_{i,j=1}^{\infty}$
in which $h_{i,j}=0$, if $i\neq j$, and the remaining morphism are the identities;

\item composition of the tuple $(g,h_{i,j})_{i,j=1}^{\infty}$ followed by 
the tuple $(g',h'_{i,j})_{i,j=1}^{\infty}$
is defined as the tuple $(g'g,\sum_k h'_{k,j}h_{i,k})_{i,j=1}^{\infty}$.
\end{itemize}

One should think about $\underline{\mathcal{A}}$ as a 
version of $\underline{\underline{\mathcal{A}}}$
with ``multiple arrows'' from the left object to the ``multiple objects'' on the right:
\[
\xymatrix@R=2mm{
& Y_1\\
X \ar[r]|{\,f_2\,}\ar[ru]^{f_1}\ar[rd]_{f_i}& Y_2\\
        & \vdots  
}
\]
These multiple 
arrows are indexed by non-negative integers, only finitely many of these arrows go to non-zero objects
and $k$ is a fixed explicit bound saying that after it arrows must terminate at the zero objects.

The category $\underline{\mathcal{A}}$ is additive, with  $\oplus$ given by
\[
(X,k,Y_i,f_i)_{i=1}^{\infty}\oplus(X',k',Y'_i,f'_i)_{i=1}^{\infty}=
(X\oplus X',\max(k,k'),Y_i\oplus Y'_i,f_i\oplus f'_i)_{i=1}^{\infty}
\]
with the evident definition on morphisms. Further, the assignment
\[
(X,k,Y_i,f_i)_{i=1}^{\infty}
\mapsto
\xymatrix{ 
X\ar[r]^/-.1cm/{\oplus_i f_i}&\oplus_i Y_i}
\]
provides an equivalence between $\underline{\mathcal{A}}$ and $\underline{\underline{\mathcal{A}}}$
which, in particular, implies that $\underline{\mathcal{A}}$ is abelian. Indeed, an inverse equivalence is
given, for example, by sending $X\overset{f}{\longrightarrow}Y$ to $(X,1,Y_i,f_i)_{i=1}^{\infty}$,
where $Y_1=Y$, $f_1=f$, $Y_i=0$ and $f_i=0$, for $i>1$. 

The original category $\mathcal{A}$ embeds
into $\underline{\mathcal{A}}$ via
\[
X\mapsto (X,0,0,0)\quad\text{ and }\quad f\colon X\to Y\mapsto (f,0). 
\]

The main point of the definition of 
$\underline{\mathcal{A}}$ is that the ``multiple objects'' 
will give us a chance to bookkeep some explicit 
direct sum constructions in the next sections.

The (diagrammatic)
{\em projective abelianization} $\overline{\mathcal{A}}$ is defined dually.

\subsection{Abelianization of finitary \texorpdfstring{$2$}{2}-categories}\label{s4.3}

The construction presented here rectifies some 
problems with strictness pointed out in \cite[Section 3.5]{MM1}.

Let $\cC$ be a finitary $2$-ca\-te\-go\-ry.
Consider the (diagrammatic) {\em injective abelianization} 
$\underline{\cC}$ of $\cC$ defined as follows:
\begin{itemize}

\item $\underline{\cC}$ has the same objects as $\cC$;

\item $\underline{\cC}(\mathtt{i},\mathtt{j})=\underline{\cC(\mathtt{i},\mathtt{j})}$
(using the above notation for injective abelianization); 

\item composition of $1$-morphisms is defined as follows:
\[
(\mathrm{F},k,\mathrm{G}_i,\alpha_i)_{i=1}^{\infty}\circ
(\mathrm{F}',k',\mathrm{G}'_i,\alpha'_i)_{i=1}^{\infty}=
(\mathrm{F}\mathrm{F}',k+k',\mathrm{H}_i,\beta_i)_{i=1}^{\infty},
\]
where
\begin{align*}
\mathrm{H}_i &=
\begin{cases}
\mathrm{F}\circ\mathrm{G}'_i, & i=1,2,\dots,k';\\
\mathrm{G}_{i-k'}\circ \mathrm{F}',  & i=k'+1,k'+2,\dots,k'+k;\\
0,& \text{else};
\end{cases} 
\\
\beta_i &=
\begin{cases}
\mathrm{id}_{\mathrm{F}}\circ_0\alpha'_i, & i=1,2,\dots,k';\\
\alpha_{i-k'}\circ_0\mathrm{id}_{\mathrm{F}'},  & i=k'+1,k'+2,\dots,k'+k;\\
0,& \text{else}.
\end{cases} 
\end{align*}

\item identity $1$-morphisms are tuples 
$(\mathbbm{1}_{\mathtt{i}},0,0,0)$, for $\mathtt{i}\in\cC$;

\item horizontal composition of $2$-morphisms is defined component-wise.

\end{itemize}
For $\mathtt{i},\mathtt{j}\in \cC$, 
mapping $\mathrm{F}$ to $(\mathrm{F},0,0,0)$ gives 
rise to an equivalence between $\cC(\mathtt{i},\mathtt{j})$ 
and the full subcategory of 
injective objects in $\underline{\cC}(\mathtt{i},\mathtt{j})$ 
which we will 
use to identify these two $2$-categories. Note that this also 
realizes $\cC$ as a 
$2$-subcategory of $\underline{\cC}$.

The (diagrammatic)
{\em projective abelianization} $\overline{\cC}$ is defined dually.

\begin{remark}\label{remark:from-co-to-usual}
\begin{enumerate}[$($a$)$]
\item \label{remark:from-co-to-usual-1}
We note that, generally, even in the case when $\cC$ is a fiat $2$-category, neither $\underline{\cC}$ nor $\overline{\cC}$ have adjunction $2$-morphisms or even a weak involution 
(see also Remark \ref{remark:egno-vs-our-stuff}). 
In fact, if $\cC$ is fiat, the weak 
involution on $\cC$ extends to a contravariant biequivalence between
$\underline{\cC}$ and $\overline{\cC}$, which 
changes the direction of both $1$- and $2$-morphisms.
\item \label{remark:from-co-to-usual-2}
In general, we do not know what kind of exactness properties the horizontal composition bifunctor $\circ$
on $\underline{\cC}$ might have. At the same time, in the case when $\cC$ is a fiat $2$-category, both
the left and the right regular actions of $\cC$ on $\underline{\cC}$ are, automatically, given by exact functors.
\end{enumerate}
\end{remark}

\subsection{Abelianization of finitary \texorpdfstring{$2$}{2}-representations}\label{s4.4}

Let $\cC$ be a finitary $2$-ca\-te\-go\-ry and 
$\mathbf{M}$ a finitary 
$2$-representation of $\cC$. Then the (diagrammatic) 
{\em injective abelianization} 
$\underline{\mathbf{M}}$ of $\mathbf{M}$ is defined by 
$\underline{\mathbf{M}}(\mathtt{i})=\underline{\mathbf{M}(\mathtt{i})}$,
for $\mathtt{i}\in\cC$ (here again we use the 
above notation for injective abelianization). 
Note that $\underline{\mathbf{M}}$ has the structure of a $2$-representation of 
$\cC$ given by the component-wise action.

Moreover, $\underline{\mathbf{M}}$ has the natural 
structure of a 
$2$-representation of $\underline{\cC}$ with the 
action defined on objects as follows:
\[
(\mathrm{F},k,\mathrm{G}_i,\alpha_i)_{i=1}^{\infty}\circ
(M,k',N_i,f_i)_{i=1}^{\infty}=
(\mathrm{F}\, M,k+k',H_i,g_i)_{i=1}^{\infty},
\]
where
\begin{align*}
H_i &=
\begin{cases}
\mathrm{F}\, N_i, & i=1,2,\dots,k';\\
\mathrm{G}_{i}\, M,  & i=k'+1,k'+2,\dots,k'+k;\\
0,& \text{else};
\end{cases} 
\\
g_i &=
\begin{cases}
\mathrm{F}\, f_i, & i=1,2,\dots,k';\\
(\alpha_{i-k'})_M,  & i=k'+1,k'+2,\dots,k'+k;\\
0,& \text{else};
\end{cases} 
\end{align*}
and with the component-wise action on morphisms.

Similarly to the above, the canonical 
embedding of ${\mathbf{M}}$ into $\underline{\mathbf{M}}$ is
a morphism of $2$-re\-presentations of 
$\cC$ and provides an equivalence between
${\mathbf{M}}$ and the $2$-representation of 
$\cC$ given by the action of 
$\cC$ on injective objects of 
$\displaystyle\coprod_{\mathtt{i}\in\ccC}\underline{\mathbf{M}}(\mathtt{i})$.

As usual, the {\em projective abelianization} 
$\overline{\mathbf{M}}$ is defined dually.

\section{Finitary \texorpdfstring{$2$}{2}-representations via (co)algebra \texorpdfstring{$1$}{1}-morphisms}\label{s5}

This section is inspired by \cite[Chapter 7]{EGNO}. 
The main goal is to provide a 
generalization of \cite[Theorem 7.10.1, Corollary 7.10.5]{EGNO}
(see also \cite{Os,EO,ENO}).

\begin{remark}\label{remark:egno-vs-our-stuff}
The framework of \cite{EGNO} is that of a tensor category, where all objects 
have duals and which is assumed to be abelian. In our situation, we have 
a $2$-category $\cC$ instead of a tensor category, moreover, we assume
that $\cC$ is enriched over additive and not necessarily abelian categories. 
The existence of dual objects in the language of \cite{EGNO} translates into 
the existence of adjoint $1$-morphisms (and, in particular, 
adjunction  $2$-morphisms) in our setup, as described in 
Section \ref{s1.5}. If our $2$-category $\cC$ 
is fiat, these adjunctions exist. However, in passing to the abelianization 
$\underline{\cC}$, we lose the existence of adjunction $2$-morphisms. 
Consequently, we are  forced to use a construction which is dual to the one 
in \cite[Section 7.9]{EGNO} (as some important 
properties fail to hold for the direct generalization of the 
construction therein).
\end{remark}

Before starting, we point out that the abstract notions 
of (co)algebra objects and their (co)module categories 
from \cite[Section 7.8]{EGNO} generalize immediately to our setting. 
The only difference is that instead of (co)algebra objects in monoidal 
categories, we consider \textit{(co)algebra $1$-morphisms} in $2$-categories. 

\subsection{Internal homs}\label{s5.1}

Recall that left exact functors between module categories are determined uniquely, 
up to isomorphism, by their action on the category of injective modules,
see for example \cite[Chapter II,\S 2]{Ba} and the dual version of it.
In particular, given an algebra $A$, any additive functor from its category 
of injective modules $A\text{-}\mathrm{inj}$
to an abelian category $\mathcal{B}$ extends uniquely, up to isomorphism,
to a left exact functor from its category of modules 
$A\text{-}\mathrm{mod}$ to $\mathcal{B}$, 
and this extension
is natural with respect to natural transformations of functors. This observation
motivates the following construction.

Let $\cC$ be a fiat $2$-category 
and $\mathbf{M}$ a transitive
$2$-re\-pre\-sen\-ta\-tion of $\cC$. For 
$\mathtt{i},\mathtt{j}\in \cC$,
$M\in \mathbf{M}(\mathtt{i})$ and 
$N\in \mathbf{M}(\mathtt{j})$, consider the unique, 
up to isomorphism, left exact functor from 
$\underline{\cC}(\mathtt{i},\mathtt{j})$ to 
the category of finite dimensional $\mathbbm{k}$-vector spaces
given by 
\begin{equation}\label{eq1}
\mathrm{F}\mapsto \mathrm{Hom}_{\mathbf{M}(\mathtt{j})}(M,\mathrm{F}\, N),
\quad\text{ for }\mathrm{F}\in \cC(\mathtt{i},\mathtt{j}).
\end{equation}
Uniqueness is due to the equivalence between  
$\cC(\mathtt{i},\mathtt{j})$ and the full subcategory of injective objects 
in $\underline{\cC}(\mathtt{i},\mathtt{j})$
as in Section \ref{s4.3}. 

Being left exact, the
functor in \eqref{eq1} is representable, that is, there 
exists a 
$1$-morphism $\underline{\mathrm{Hom}}(N,M)\in \underline{\cC}(\mathtt{i},\mathtt{j})$, 
unique up to isomorphism, with an isomorphism
\begin{equation}\label{eq2}
\mathrm{Hom}_{\mathbf{M}(\mathtt{j})}(M,\mathrm{F}\, N)\cong
\mathrm{Hom}_{\underline{\ccC}(\mathtt{i},\mathtt{j})}(\underline{\mathrm{Hom}}(N,M),\mathrm{F}),
\quad\text{ for any }\mathrm{F}\in \cC(\mathtt{i},\mathtt{j}).
\end{equation}
The $1$-morphism $\underline{\mathrm{Hom}}(N,M)$ 
is called the {\em internal hom from $N$ to $M$}. 

In fact, the isomorphism from \eqref{eq2} also exists 
for all $\mathrm{F}\in\underline{\cC}(\mathtt{i},\mathtt{j})$:

\begin{lemma}\label{lemn2}
There is an isomorphism 
\begin{equation}\label{eq2-1}
\mathrm{Hom}_{\underline{\mathbf{M}}(\mathtt{j})}(M,\mathrm{F}\, N)\cong
\mathrm{Hom}_{\underline{\ccC}(\mathtt{i},\mathtt{j})}(\underline{\mathrm{Hom}}(N,M),\mathrm{F}),
\quad\text{ for any }\mathrm{F}\in \underline{\cC}(\mathtt{i},\mathtt{j}).
\end{equation}
\end{lemma}

\begin{proof}
As the canonical embedding of ${\cC}(\mathtt{i},\mathtt{j})$ into the subcategory of 
injective objects of $\underline{\cC}(\mathtt{i},\mathtt{j})$ is an 
equivalence, see Section \ref{s4.3}, the claim follows from \eqref{eq2} by 
standard arguments using left exactness and the Five Lemma.
\end{proof}

\begin{lemma}\label{lem2}
If $M=N$, then $\mathtt{i}=\mathtt{j}$ and $A^N=\underline{\mathrm{Hom}}(N,N)$ has the  
structure of a coalgebra $1$-morphism in $\underline{\cC}(\mathtt{i},\mathtt{i})$. 
\end{lemma}

\begin{proof}
We have the coevaluation map 
\[
\mathrm{coev}_{N,N}\colon N\to  \underline{\mathrm{Hom}}(N,N)\, N
\]
given by the image of $\mathrm{id}_{\underline{\mathrm{Hom}}(N,N)}$ under (the special
case of) the isomorphism 
\[
\mathrm{Hom}_{\mathbf{M}(\mathtt{i})}(N,\underline{\mathrm{Hom}}(N,N)\, N)\cong
\mathrm{Hom}_{\underline{\ccC}(\mathtt{i},\mathtt{i})}(\underline{\mathrm{Hom}}(N,N),
\underline{\mathrm{Hom}}(N,N))
\]
in \eqref{eq2-1}.
Following the arguments in \cite[Equation (7.29)]{EGNO}, 
this gives rise to the comultiplication morphism
\[
\underline{\mathrm{Hom}}(N,N)\to \underline{\mathrm{Hom}}(N,N)\circ \underline{\mathrm{Hom}}(N,N). 
\]

Similarly to \cite[Equation (7.30)]{EGNO}, the 
counit map is defined using the image of
$\mathrm{id}_N$ under (the special case of) the isomorphism 
\[
\mathrm{Hom}_{\mathbf{M}(\mathtt{i})}(N,\mathbbm{1}_{\mathtt{i}}\, N)\cong
\mathrm{Hom}_{\underline{\ccC}(\mathtt{i},\mathtt{i})}(\underline{\mathrm{Hom}}(N,N),
\mathbbm{1}_{\mathtt{i}})
\]
in \eqref{eq2}.
Finally, similarly to \cite[Section 7.9]{EGNO}, 
coassociativity and the other axioms are checked by direct computation. 
\end{proof}

We note that existence of adjoint $1$-morphisms (and adjunction $2$-morphisms) 
for the (analog of the) whole of $\underline{\cC}$ is not  used in \cite[Equations (7.29) and (7.30)]{EGNO}.

\subsection{The categories of comodules}\label{s5.2}

Let $\mathrm{comod}_{\underline{\ccC}}(A^N)$ denote the 
category of right \textit{$A^N$-comodule $1$-morphisms} in 
$\displaystyle\coprod_{\mathtt{j}\in\ccC}\underline{\cC}(\mathtt{i},\mathtt{j})$ 
and
$\mathrm{inj}_{\underline{\ccC}}(A^N)$ denote the \textit{subcategory of 
injective objects} in 
$\mathrm{comod}_{\underline{\ccC}}(A^N)$. Then we have 
the functor
\begin{equation}\label{eq3}
\begin{array}{rcl}
\Theta\colon\displaystyle\coprod_{\mathtt{j}\in\ccC}\mathbf{M}(\mathtt{j})&\to & 
\mathrm{comod}_{\underline{\ccC}}(A^N),\\
M&\mapsto& \underline{\mathrm{Hom}}(N,M),
\end{array}
\end{equation}
with the evident assignment for morphisms.

\begin{lemma}\label{lem3}
The functor $\Theta$ (weakly) commutes with the action of $\cC$ and defines a morphism of $2$-representations.
\end{lemma}

\begin{proof}
We need to show that 
\begin{equation}\label{eq4}
\underline{\mathrm{Hom}}(N,\mathrm{F}\, M) \cong \mathrm{F}\, \underline{\mathrm{Hom}}(N,M),
\end{equation}
for any $1$-morphism $\mathrm{F}$ in $\cC(\mathtt{i},\mathtt{j})$. 
By the property of $\cC$ being fiat, \eqref{eq2} and 
uniqueness of the representing object up to 
isomorphism, we have 
\[
\begin{array}{rcl}
\mathrm{Hom}_{\underline{\ccC}(\mathtt{i},\mathtt{j})}(\underline{\mathrm{Hom}}(N,\mathrm{F}\, M),\mathrm{G})
&\cong& 
\mathrm{Hom}_{\mathbf{M}(\mathtt{j})}(\mathrm{F}\, M,\mathrm{G}\, N)
\\ 
 &\cong& 
\mathrm{Hom}_{\mathbf{M}(\mathtt{i})}(M,\mathrm{F}^{*}\circ\mathrm{G}\, N)
\\
 &\cong& 
\mathrm{Hom}_{\underline{\ccC}(\mathtt{i},\mathtt{i})}(\underline{\mathrm{Hom}}(N,M),\mathrm{F}^{*}\circ\mathrm{G})
\\
 &\cong& 
\mathrm{Hom}_{\underline{\ccC}(\mathtt{i},\mathtt{j})}(\mathrm{F}\circ \underline{\mathrm{Hom}}(N,M),\mathrm{G}),
\end{array}
\]
for any $\mathrm{G} \in \underline{\cC}(\mathtt{i},\mathtt{j})$, 
and the first statement follows.

For the second, we need to check coherence, in other words, we need to check that for any $\mathrm{G} \in \underline{\cC}(\mathtt{i},\mathtt{j}),\mathrm{F} \in \underline{\cC}(\mathtt{j},\mathtt{k}) , \mathrm{H} \in \underline{\cC}(\mathtt{i},\mathtt{k})$, the composite of isomorphisms 
\begin{equation*}\begin{split}
\mathrm{Hom}_{\underline{\ccC}(\mathtt{i},\mathtt{k})}&(\underline{\mathrm{Hom}}(N,\mathrm{F} \mathrm{G}M),\mathrm{H} )
\!\overset{\phi_1}{\longrightarrow}\! \mathrm{Hom}_{\mathbf{M}(\mathtt{k})}(\mathrm{F} \mathrm{G}M,\mathrm{H} N)         
\!\overset{\phi_2}{\longrightarrow}\! \mathrm{Hom}_{\mathbf{M}(\mathtt{i})}(M, (\mathrm{F} \mathrm{G})^*\mathrm{H} N)                 \\ & 
\overset{\phi_3}{\longrightarrow} \mathrm{Hom}_{\underline{\ccC}(\mathtt{i},\mathtt{i})}(\underline{\mathrm{Hom}}(N,M), (\mathrm{F} \mathrm{G})^*\mathrm{H} )              
\overset{\phi_4}{\longrightarrow}\mathrm{Hom}_{\underline{\ccC}(\mathtt{i},\mathtt{k})}(\mathrm{F} \mathrm{G}\underline{\mathrm{Hom}}(N,M), \mathrm{H} )        
\end{split}\end{equation*}
is equal to the composite of isomorphisms
\begin{equation*}\begin{split}
\mathrm{Hom}_{\underline{\ccC}(\mathtt{i},\mathtt{k})}&(\underline{\mathrm{Hom}}(N,\mathrm{F}\mathrm{G}M),\mathrm{H} )
\overset{\psi_1}{\longrightarrow} \mathrm{Hom}_{\mathbf{M}(\mathtt{k})}(\mathrm{F}\mathrm{G}M,\mathrm{H} N)   
\overset{\psi_2}{\longrightarrow} \mathrm{Hom}_{\mathbf{M}(\mathtt{j})}(\mathrm{G}M,\mathrm{F}^*\mathrm{H} N)  \\ &  
\overset{\psi_3}{\longrightarrow} \mathrm{Hom}_{\underline{\ccC}(\mathtt{i},\mathtt{j})}(\underline{\mathrm{Hom}}(N,\mathrm{G}M),\mathrm{F}^*\mathrm{H} )
\overset{\psi_4}{\longrightarrow} \mathrm{Hom}_{\mathbf{M}(\mathtt{j})}(\mathrm{G}M,\mathrm{F}^*\mathrm{H} N)  \\&
\overset{\psi_5}{\longrightarrow} \mathrm{Hom}_{\mathbf{M}(\mathtt{i})}(M,\mathrm{G}^*\mathrm{F}^*\mathrm{H} N)
\overset{\psi_6}{\longrightarrow} \mathrm{Hom}_{\underline{\ccC}(\mathtt{i},\mathtt{i})}(\underline{\mathrm{Hom}}(N,M), \mathrm{G}^*\mathrm{F}^*\mathrm{H} )     \\&
\overset{\psi_7}{\longrightarrow} \mathrm{Hom}_{\underline{\ccC}(\mathtt{i},\mathtt{j})}(\mathrm{G}\underline{\mathrm{Hom}}(N,M), \mathrm{F}^*\mathrm{H} )    
\overset{\psi_8}{\longrightarrow} \mathrm{Hom}_{\underline{\ccC}(\mathtt{i},\mathtt{k})}(\mathrm{F} \mathrm{G}\underline{\mathrm{Hom}}(N,M), \mathrm{H} ).   
\end{split}\end{equation*}

Since $\phi_1=\psi_1$ and $\psi_3$ and $\psi_4$ are mutual inverses, it suffices to check that $\phi_4\phi_3\phi_2 = \psi_8\psi_7\psi_6\psi_5\psi_2$, that is, that the solid part of the 
following diagram 
\[
\xymatrix@C.75cm{
\mathrm{Hom}_{\mathbf{M}(\mathtt{k})}(\mathrm{F}\mathrm{G}M,\mathrm{H} N) 
\ar[rrr]^{\psi_2}\ar[d]_{\phi_2}& &&\mathrm{Hom}_{\mathbf{M}(\mathtt{j})}(\mathrm{G}M,\mathrm{F}^*\mathrm{H} N)  
  \ar[d]^{\psi_5} \\
 \mathrm{Hom}_{\mathbf{M}(\mathtt{i})}(M, (\mathrm{F} \mathrm{G})^*\mathrm{H} N)  \ar[d]_{\phi_3} \ar@{<-->}[rrr] &&& \mathrm{Hom}_{\mathbf{M}(\mathtt{i})}(M,\mathrm{G}^*\mathrm{F}^*\mathrm{H} N)\ar[d]^{\psi_6} \\
\mathrm{Hom}_{\underline{\ccC}(\mathtt{i},\mathtt{i})}(\underline{\mathrm{Hom}}(N,M), (\mathrm{F} \mathrm{G})^*\mathrm{H} )     \ar[d]_{\phi_4}&&&\ar@{<-->}[lll] \mathrm{Hom}_{\underline{\ccC}(\mathtt{i},\mathtt{i})}(\underline{\mathrm{Hom}}(N,M), \mathrm{G}^*\mathrm{F}^*\mathrm{H} )     \ar[d]^{\psi_7}\\
\mathrm{Hom}_{\underline{\ccC}(\mathtt{i},\mathtt{k})}(\mathrm{F} \mathrm{G}\underline{\mathrm{Hom}}(N,M), \mathrm{H} )&&& \mathrm{Hom}_{\underline{\ccC}(\mathtt{i},\mathtt{j})}(\mathrm{G}\underline{\mathrm{Hom}}(N,M), \mathrm{F}^*\mathrm{H} )    \ar[lll]^{\psi_8}}
\]commutes.
Uniqueness of adjoints up to unique isomorphism (cf. \cite[Proposition 2.10.5]{EGNO}) gives us a unique isomorphism $(\mathrm{F} \mathrm{G})^* \cong \mathrm{G}^*\mathrm{F}^*$, which fills in the dashed arrows in the diagram, making the top and bottom squares commutative. Commutativity of the middle square is due to naturality of the isomorphism \eqref{eq2-1} from Lemma \ref{lemn2}. Coherence of $\Theta$ and hence the claim that it is a morphism of $2$-representations follows.
\end{proof}

\begin{lemma}\label{lem5}
For any $1$-morphism $\mathrm{F}$ in $\cC$ and 
any $X\in \mathrm{comod}_{\underline{\ccC}}(A^N)$,
we have an isomorphism
\begin{equation}\label{eq5}
\mathrm{Hom}_{\mathrm{comod}_{\underline{\ccC}}(A^N)}(X,\mathrm{F}\, A^N)\cong 
\mathrm{Hom}_{\underline{\ccC}}(X,\mathrm{F}).
\end{equation}
\end{lemma}

\begin{proof}
Postcomposing $\alpha\colon X\to \mathrm{F}\, A^N $ with 
the map $\mathrm{F}\, A^N\to \mathrm{F}$ obtained by
applying $\mathrm{F}$ to the counit morphism of $A^N$, we get 
a $\mathbbm{k}$-linear map from the left-hand 
side to the right-hand side of \eqref{eq5}.

Precomposing the evaluation of $\beta\colon X\to \mathrm{F}$ 
at $A^N$ with the map 
$X\to X\, A^N$ coming from the comodule structure on $A^N$, we get 
a $\mathbbm{k}$-linear map from the 
right-hand side to the left-hand side 
of \eqref{eq5}. 

It is straightforward to check that these two maps are 
inverse to each other.
\end{proof}

\begin{lemma}\label{lem4}
The functor $\Theta$ factors over the 
inclusion 
$\mathrm{inj}_{\underline{\ccC}}(A^N)\hookrightarrow\mathrm{comod}_{\underline{\ccC}}(A^N)$.
\end{lemma}

\begin{proof}
If $M=\mathrm{F}\,N$, for some 
$1$-morphism $\mathrm{F}$ in $\cC$, 
then, from \eqref{eq4}, we have
\[
\underline{\mathrm{Hom}}(N,\mathrm{F}\, N)\cong\mathrm{F}\, A^N.
\]

Next we claim that $A^N$ is injective in $\mathrm{comod}_{\underline{\ccC}}(A^N)$.
From Lemma \ref{lem5}, it follows that
\[
\mathrm{Hom}_{\mathrm{comod}_{\underline{\ccC}}(A^N)}({}_-,A^N)\cong 
\mathrm{Hom}_{\underline{\ccC}}({}_-,\mathbbm{1}_{\mathtt{i}}).
\]
Hence, injectivity of $A^N$ as a comodule reduces to injectivity of $\mathbbm{1}_{\mathtt{i}}$ as 
an object of $\underline{\cC}$, 
which holds by construction of the injective abelianization, cf. Section \ref{s4.3}. 

Due to fiatness of $\cC$, it follows that $\mathrm{F}\, A^N$
is injective in $\mathrm{comod}_{\underline{\ccC}}(A^N)$.
By transitivity of $\mathbf{M}$, any $M$ is isomorphic to a direct summand of 
some $\mathrm{F}\,N$. The claim follows.
\end{proof}

\subsection{Transitive \texorpdfstring{$2$}{2}-representations and comodule categories}\label{s5.3}

Recall that $\cC$ denotes a fiat $2$-category 
and $\mathbf{M}$ a transitive
$2$-re\-pre\-sen\-ta\-tion of $\cC$.

\begin{theorem}\label{thm7}
Let $N\in\mathbf{M}(\mathtt{i})$ be non-zero. Then the functor $\Theta$ gives rise to 
an equivalence of $2$-representations of 
$\cC$ between $\underline{\mathbf{M}}$ and $\mathrm{comod}_{\underline{\ccC}}(A^N)$.
This equivalence restricts to an equivalence of $2$-representations of 
$\cC$ between ${\mathbf{M}}$ and $\mathrm{inj}_{\underline{\ccC}}(A^N)$.
\end{theorem}

\begin{proof}
For any objects of the form $\mathrm{F}N, \mathrm{G}N$ in 
$\displaystyle\mathcal{M}=\coprod_{\mathtt{j}\in\ccC}\mathbf{M}(\mathtt{j})$, we have 
\begin{equation*}
\begin{split}
\mathrm{Hom}_{\mathrm{comod}_{\underline{\ccC}}(A^N)}(\Theta(\mathrm{F}N), \Theta(\mathrm{G}N))&\cong 
\mathrm{Hom}_{\mathrm{comod}_{\underline{\ccC}}(A^N)}(\mathrm{F}\,A^N,\mathrm{G}\,A^N)\\
&\cong \mathrm{Hom}_{\underline{\ccC}}(\mathrm{F}\,A^N,\mathrm{G})\\
&\cong \mathrm{Hom}_{\underline{\ccC}}(A^N, \mathrm{F}^*\mathrm{G})\\
&\cong \mathrm{Hom}_{\mathcal{M}}(N, \mathrm{F}^*\mathrm{G}\,N)\\
&\cong \mathrm{Hom}_{\mathcal{M}}(\mathrm{F}\,N, \mathrm{G}\,N),
\end{split}
\end{equation*}
where the first isomorphism uses Lemma \ref{lem3} and the 
second isomorphism uses Lemma \ref{lem5}.
From transitivity of $\mathbf{M}$, we deduce that, for any 
$N_1,N_2\in \mathcal{M}$, we have 
\[ 
\mathrm{Hom}_{\mathrm{comod}_{\underline{\ccC}}(A^N)}(\Theta(N_1),\Theta(N_2))\cong
\mathrm{Hom}_{\mathcal{M}}(N_1,N_2).
\]

The functor $\underline{\mathrm{Hom}}(N,{}_-)$ extends 
uniquely (up to isomorphism) to a left exact 
functor from 
$\displaystyle\underline{\mathcal{M}}=\coprod_{\mathtt{j}\in\ccC}\underline{\mathbf{M}}(\mathtt{j})$ 
to $\underline{\mathrm{inj}}_{\underline{\ccC}}(A^N)$, 
where the latter category is equivalent 
to $\mathrm{comod}_{\underline{\ccC}}(A^N)$, as usual (cf. Section \ref{s4}).
 
Considering injective resolutions of 
$N_1,N_2 \in \underline{\mathcal{M}}$ by objects in $\mathcal{M}$, 
using left exactness and 
arguments similar to the one in the proof of \cite[Theorem 7.10.1(2)]{EGNO}, 
one checks that, indeed, 
\[ 
\mathrm{Hom}_{\mathrm{comod}_{\underline{\ccC}}(A^N)}(\Theta(N_2),\Theta(N_1))\cong
\mathrm{Hom}_{\underline{\mathcal{M}}}(N_1,N_2),
\]
for any $N_1,N_2\in \underline{\mathcal{M}}$.
Consequently, $\Theta\colon \underline{\mathcal{M}}\to \underline{\mathrm{inj}}_{\underline{\ccC}}(A_N)$ 
is full and faithful and it remains to 
show that it is essentially surjective.

To prove that $\Theta$ is essentially surjective, it suffices to show 
that $\Theta$ maps an injective cogenerator of 
$\underline{\mathbf{M}}$ to an injective cogenerator of 
$\mathrm{comod}_{\underline{\ccC}}(A^N)$. In other words, it suffices to show that any injective $A^N$-comodule is 
isomorphic to a direct summand
of a comodule of the form $\mathrm{F}\, A^N$, for some 
$1$-morphism $\mathrm{F}$ in $\cC$. To prove the latter,
we just need to show that any $A^N$-comodule injects into a 
comodule of the form $\mathrm{F}\, A^N$, 
for some $1$-morphism $\mathrm{F}$. These claims are, basically, 
the injective versions of \cite[Exercises 7.8.14 and 7.8.15]{EGNO}.

Let $X\in \mathrm{comod}_{\underline{\ccC}}(A^N)$. Then the 
coaction map
$X\mapsto X\, A^N$ is a homomorphism of $A^N$-comodules. This 
homomorphism is 
injective as its postcomposition with the evaluation 
$X\, A^N\to X$ at $X$ of the
counit morphism for $A^N$ is the identity. 
Further, if $X$ is of the form $(Y,k,Z_i,\alpha_i)$, then 
we have a natural injection of 
$X\, A^N$ into $Y\, A^N$ given by the tuple $(\mathrm{id}_Y\circ_0 \mathrm{id}_{A^N},0)$. 
Now we note that the $A^N$-comodule $Y\, A^N$ has the 
necessary form. 
This shows that $\Theta$ is essentially surjective. As we have already 
established that it is full and faithful, 
the first claim of the theorem follows. 

The second claim of the theorem follows from the first one and Lemma \ref{lem4}.
\end{proof}

Different choices of $N$ lead to 
Morita--Takeuchi equivalent coalgebra $1$-morphisms, 
not to isomorphic ones, as we will explain in 
Corollary \ref{cor:Mor-Tak}. For the remainder of this section, we fix some non-zero 
object $N$ in $\underline{\mathcal{M}}$.

\subsection{Transitive \texorpdfstring{$2$}{2}-representations and module categories}\label{s5.4}

We get a ``dual version'' of Theorem \ref{thm7} as well. For an 
algebra $1$-morphism  $A$ in $\overline{\cC}$, denote by $\mathrm{mod}_{\overline{\ccC}}(A)$ the 
category of right \textit{$A$-module $1$-mor\-phisms} in $\overline{\cC}$
and by $\mathrm{proj}_{\overline{\ccC}}(A)$ the 
\textit{subcategory of projective $A$-module $1$-morphisms} in $\overline{\cC}$.

\begin{corollary}\label{cor-72}
There exists an algebra $1$-morphism $A_N$ in $\overline{\cC}$ and an 
equivalence of $2$-rep\-re\-sen\-ta\-ti\-ons of $\cC$ between $\overline{\mathbf{M}}$ and 
$\mathrm{mod}_{\overline{\ccC}}(A_N)$, which restricts to an equivalence of 
$2$-representations of $\cC$ between ${\mathbf{M}}$ and $\mathrm{proj}_{\overline{\ccC}}(A_N)$.
\end{corollary}

\begin{proof}
The weak involution ${}^*$ of the fiat category $\cC$ gives rise to a duality
(on the level of both $1$- and $2$-morphisms)
between $\underline{\cC}$ and $\overline{\cC}$.
This swaps coalgebra objects with algebra objects and comodules with modules.
Therefore all claims follow from Theorem \ref{thm7} using this duality.
\end{proof}

Unfortunately, we do not see how to prove 
Corollary \ref{cor-72} directly, as in Theorem \ref{thm7} 
we heavily rely on {\em left} exactness of all 
constructions, in particular, of the internal hom bifunctor.

\begin{corollary}\label{cor9}
If $\mathbf{M}$ is simple transitive, then $A_N$ is simple.
\end{corollary}

\begin{proof}
If $J$ is a non-trivial two-sided ideal of $A_N$, then 
all morphisms of the form $X\,\alpha$, where $\alpha\in J$ and  
$X\in\mathrm{proj}_{\overline{\ccC}}(A_N)$, generate 
a non-trivial $\cC$-stable ideal in $\mathrm{proj}_{\overline{\ccC}}(A_N)$, 
contradicting simple transitivity of the action of 
$\cC$ on $\mathrm{proj}_{\overline{\ccC}}(A_N)$. The fact that the ideal
in question is non-trivial follows directly from nilpotency of the radical of any
finite dimensional algebra.
\end{proof}

\begin{corollary}\label{cor11}
For $\mathtt{i}\in \cC$, consider the 
endomorphism $2$-category $\cA$ of $\mathtt{i}$ in $\cC$ (in particular, 
$\cA(\mathtt{i},\mathtt{i})=\cC(\mathtt{i},\mathtt{i})$). Then there is a natural
bijection between the equivalence classes of simple transitive $2$-representations of
$\cA$ and the equivalence classes of simple transitive $2$-representations of
$\cC$ having a non-trivial value at $\mathtt{i}$.
\end{corollary}

\begin{proof}
We start by noting that $\cA$ inherits 
the structure of a fiat $2$-category from $\cC$.
Now, consider a simple transitive $2$-representation 
$\mathbf{M}$ of $\cC$ such that 
$\mathbf{M}(\mathtt{i})\neq 0$. Then we have the corresponding algebra 
$1$-morphism $A_N$ 
in $\overline{\cC}$ given by 
Corollary \ref{cor-72}. By construction, $A_N$ is also 
in $\overline{\cA}$ and,
by Corollary \ref{cor-72}, thus gives rise to simple 
transitive $2$-representation of $\cA$. 
(This $2$-representation is just the 
restriction of $\mathbf{M}$ to $\cA$.) 
As any algebra $1$-morphism in $\overline{\cA}$ is, at the 
same time, an algebra $1$-morphism in 
$\overline{\cC}$, we obtain that the above 
correspondence is bijective.
\end{proof}

\subsection{Some remarks and a bonus observation}\label{s5.5}

\begin{remark}\label{remark13}
{\rm  
Let $\cC$ be a fiat $2$-category, $\mathbf{M}$ a transitive $2$-re\-pre\-sen\-ta\-tion of 
$\cC$ and $N$ an object in $\mathbf{M}(\mathtt{i})$. In all examples we know, 
there exists an algebra $1$-morphism $A$ in $\cC$ such that  the algebra $1$-morphism 
$A_N$ is a quotient $1$-morphism of $A$. However, we do not know whether this
property holds in full generality.
}
\end{remark}

The following bonus observation is inspired by \cite[Theorem 2]{KMMZ}.

\begin{proposition}\label{lem501}
Assume that $\cC$ is a fiat $2$-category,  $\mathbf{M}$ a simple transitive
$2$-rep\-re\-sen\-tation of $\cC$ and $N$ an object in $\mathbf{M}(\mathtt{i})$. 
Then $A_N$ is self-injective.
\end{proposition}

Before we give the proof of Proposition \ref{lem501}, recall the notion 
of an \textit{apex of a 
transitive $2$-representation}, given in \cite[Section 3]{CM}, which is the
maximal, with respect to the two-sided order, $\mathcal{J}$-cell which does
not annihilate the $2$-representation.

\begin{proof}
Let $\mathcal{J}$ be the apex 
of $\mathbf{M}$. Analogously as in \cite[Theorem 2]{KMMZ}, 
one can prove that
$1$-morphisms in 
$\mathcal{J}$ send simple objects in $\overline{\mathbf{M}}$ to projective objects.
A dual argument gives that $1$-morphisms in $\mathcal{J}$ 
send simple objects in $\overline{\mathbf{M}}$ to injective objects. The 
claim follows.
\end{proof}

\begin{remark}\label{remark:inj-proj}
We could alternatively have used $\underline{\mathbf{M}}$ instead
of $\overline{\mathbf{M}}$ in the proof of Proposition \ref{lem501}. As the
referee pointed out to us, it is an interesting question
whether $\underline{\mathbf{M}}$ and $\overline{\mathbf{M}}$ are equivalent 
(as abelian $2$-representations).
However, in full generality the answer is negative. For example,
if $\cC$ is not a fiat $2$-category, then 
$\underline{\mathbf{M}}$ and $\overline{\mathbf{M}}$ are not equivalent 
due to different sides of half-exactness
for the $2$-action of $\cC$ on them.
On the other hand, thanks to the assertion of Proposition \ref{lem501},
such an equivalence
seems to have a chance to be true in
case $\cC$ is fiat and $\mathbf{M}$ is simple transitive.
\end{remark}

\begin{remark}\label{remark:graded}
All results in this section can also 
be formulated and proven in the graded 
setup as, for example,
considered in \cite[Section 7]{MM3} 
or \cite[Section 3]{MT}.
\end{remark}


\section{Morita--Takeuchi theory for (co)algebra \texorpdfstring{$1$}{1}-morphisms}
\label{section:Morita-Takeuchi}

Morita theory \cite{Mo} explains that, under certain conditions, an equivalence 
of categories $\mathrm{mod}(A)\cong 
\mathrm{mod}(B)$, where $A$ and $B$ are algebras, can 
be given in terms of tensoring with an $A$-$B$ bimodule. 
Takeuchi \cite{Ta} ``dualized'' 
Morita theory for coalgebras, comodules and bicomodules. 
In this section, 
we explain how Morita(--Takeuchi) theory extends to (co)algebra, 
(co)module and bi(co)module $1$-morphisms in finitary and fiat 
$2$-categories.

\subsection{The bicomodule story}\label{s6.1}

Let $\cC$ be a finitary $2$-category and 
$A, B$ two coalgebra $1$-morphisms in $\underline{\cC}$. 
Suppose that 
$M={}_A M{}_B$ and $N={}_BN{}_A$ are bicomodule $1$-mor\-phisms 
in $\underline{\cC}$ over $A$ and $B$, with the left and 
right coaction $2$-morphisms on the indicated sides denoted 
by $\lambda_M, \lambda_N$ and $\rho_M,\rho_N$, respectively. 
Recall \cite[\S 0]{Ta} that the {\em cotensor product} $M\square_B N$ 
is defined as the kernel of the $2$-morphism  
\[
M \circ N \xrightarrow{\rho_M \circ 
\mathrm{id}_N - \mathrm{id}_M\circ \lambda_N} M \circ B\circ N.
\] 
This is an $A$-$A$ bicomodule $1$-morphism in $\underline{\cC}$. 
Similarly, one defines the $B$-$B$ bicomodule 
$1$-morphism $N\square_A M$.

The bicomodule $1$-morphisms $M$ and $N$ induce functors:  
\[
{}_- \;\square_A M \colon \mathrm{comod}_{\underline{\ccC}}(A) 
\to  \mathrm{comod}_{\underline{\ccC}}(B)\quad\text{and}\quad 
{}_-\;\square_B N   \colon \mathrm{comod}_{\underline{\ccC}}(B) \to  \mathrm{comod}_{\underline{\ccC}}(A).
\]
Given two $A$-$B$ bicomodule $1$-morphisms $M_1$ and $M_2$, 
there is a bijection 
\[
\left(\alpha\colon M_1 \to M_2\right)
\mapsto\left(
{}_-\; \square\; \alpha\colon {}_-\square_A M_1 \to {}_-\square_A M_2\right)
\]
between 
bicomodule $2$-morphisms and the associated 
natural transformations. 

It is easy to see that the cotensor product is associative 
(up to a canonical isomorphism) and that $\lambda_M$ and $\rho_M$ 
define $A$-$B$ bicomodule $2$-isomorphisms 
\[
M\overset{\cong}{\longrightarrow} A\;\square_{A} M\quad\text{and}\quad 
M \overset{\cong}{\longrightarrow}M\;\square_B B, 
\]
whose inverses are given by the counit $2$-morphisms. 

We say that a bicomodule $1$-morphism is {\em biinjective} 
if it is injective as a left comodule $1$-morphism and 
as a right comodule $1$-morphism (but 
not necessarily as a bicomodule $1$-morphism).  

\begin{theorem}\label{thm:Mor-Tak} 
As $2$-representations of 
$\underline{\cC}$, $\mathrm{comod}_{\underline{\ccC}}(A)$ and 
$\mathrm{comod}_{\underline{\ccC}}(B)$ 
are equivalent if and only if there exist biinjective 
bicomodule $1$-morphisms    
$M={}_A M_B$ and $N={}_BN_A$  in $\underline{\cC}$ and 
bicomodule $2$-isomorphisms 
\[
f\colon A \overset{\cong}{\longrightarrow} 
M\square_B N\quad\text{and}\quad g\colon B 
\overset{\cong}{\longrightarrow} N\square_A M, 
\]
such that the diagrams 
\begin{gather}\label{equation:squares}
\begin{aligned}
\xymatrix@C=12mm{
 M    \ar[r]^/-.25cm/{\rho_M}\ar[d]_{\lambda_M} 
 &       M\square_B B \ar[d]^{\mathrm{id}_M\square g}\\
A\square_A M    \ar[r]_/-.25cm/{f\square \mathrm{id}_M} 
& M\square_B N\square_A M
}
\quad\quad
\xymatrix@C=12mm{
 M    \ar[r]^/-.25cm/{\rho_N}\ar[d]_{\lambda_N} 
 &       N\square_A A \ar[d]^{\mathrm{id}_N\square f}\\
B\square_B N    \ar[r]_/-.25cm/{g\square \mathrm{id}_N} 
& N\square_A M\square_B N
}
\end{aligned}
\end{gather}
commute. 
\end{theorem}

\begin{proof}
The ``if'' part is clear, since the 
equivalence $\mathrm{comod}_{\underline{\ccC}}(A)
\cong\mathrm{comod}_{\underline{\ccC}}(B)$ 
is defined by ${}_-\; \square_A M$ and 
its inverse ${}_-\; \square_B N$. 

Next, we prove the ``only if'' part. 
This follows as in \cite[\S 2 and \S 3]{Ta}, 
with only minor changes in the details of the proofs. 
Let 
$F\colon \mathrm{comod}_{\underline{\ccC}}(A) 
\to \mathrm{comod}_{\underline{\ccC}}(B)$ be an exact 
functor which intertwines the 
two $2$-actions of $\underline{\cC}$. Then $F(A)$ is an 
$A$-$B$ bicomodule $1$-morphism and  
$F({}_-)\cong{}_-\; \square_A F(A)$, following \cite[Proposition 2.1]{Ta}. 
Note that the proof is identical, except that we replace the 
preservation of direct sums by the intertwining property in the 
first sentence of Takeuchi's proof. 
Similarly, if $G$ is the inverse of $F$, then $G(B)$ is an $B$-$A$ 
bicomodule $1$-morphism and $G({}_-)\cong {}_-\; \square_B G(B)$. 

The natural isomorphisms 
\[
\mathrm{id}_{\mathrm{comod}_{\underline{\ccC}}(B)} 
\overset{\cong}{\longrightarrow} FG
\quad\text{and}\quad 
\mathrm{id}_{\mathrm{comod}_{\underline{\ccC}}(A)} 
\overset{\cong}{\longrightarrow} GF
\] 
give rise to $2$-isomorphisms  
$f\colon A \to F(A)\square_B G(B)$ and 
$g\colon B\to G(B)\square_A F(A)$ such that the 
squares in \eqref{equation:squares} commute. The 
same natural isomorphisms imply that $F$ and $G$ are biadjoint, 
so both functors are exact and send injectives to injectives. 
Therefore, $F(A)$ and $G(B)$ are biinjective 
(cf. \cite[Theorem 2.5]{Ta}).  
\end{proof}

The following corollary follows directly from 
Theorems \ref{thm7} and \ref{thm:Mor-Tak}.

\begin{corollary}\label{cor:Mor-Tak}
Let $\cC$ be a finitary $2$-category and 
$\mathbf{M}$ a transitive 
$2$-representation of $\cC$. 
Furthermore, let $N_1,N_2\in \mathbf{M}(\mathtt{i})$ 
be two objects as in Theorem \ref{thm7}. 
Then $\underline{\mathrm{Hom}}(N_1,N_2)$ is a 
biinjective  
$\underline{\mathrm{Hom}}(N_2,N_2)$-$\underline{\mathrm{Hom}}(N_1,N_1)$ 
bicomodule $1$-morphism in $\underline{\cC}$ and 
$\underline{\mathrm{Hom}}(N_2,N_1)$ is a 
biinjective $\underline{\mathrm{Hom}}(N_1,N_1)$-$\underline{\mathrm{Hom}}(N_2,N_2)$ 
bicomodule $1$-morphism in $\underline{\cC}$, such that 
\begin{align*}
\underline{\mathrm{Hom}}(N_1,N_1)&\cong 
\underline{\mathrm{Hom}}(N_2,N_1)
\square_{\underline{\mathrm{Hom}}(N_2,N_2)} \underline{\mathrm{Hom}}(N_1,N_2),\\
\underline{\mathrm{Hom}}(N_2,N_2)&\cong
\underline{\mathrm{Hom}}(N_1,N_2)
\square_{\underline{\mathrm{Hom}}(N_1,N_1)} \underline{\mathrm{Hom}}(N_2,N_1).
\end{align*}
\end{corollary}

\subsection{The bimodule story}\label{s6.2}

As already noted in Remark \ref{remark:from-co-to-usual}, if $\cC$ is fiat, 
the involution ${}^*$ on $\cC$ gives rise to an 
equivalence 
between $\underline{\cC}$ and $\overline{\cC}$, 
which sends injective $1$-morphisms to projective 
$1$-morphisms. Using this fact, 
we see that Theorem \ref{thm:Mor-Tak} implies a ``dual 
theorem'' for algebra $1$-morphisms in $\overline{\cC}$, which 
we state below without further proof.

For this purpose, 
let $M$ be an $A$-$B$ bimodule $1$-morphism in $\overline{\cC}$, 
with left action $2$-mor\-phism $\lambda_M\colon A\circ M\to M$ 
and right action $2$-morphism $\rho_M\colon M\circ B \to M$. Similarly, let $N$ 
be a $B$-$A$ bimodule $1$-morphism. As before, 
define $M\circ_B N$ to be the cokernel of the $2$-morphism  
\[
M \circ B\circ N \xrightarrow{\rho_M \circ \mathrm{id}_N - \mathrm{id}_M\circ \lambda_N} M \circ N.
\] 
Define $N\circ_A M$ similarly. 
Note that $\lambda_M$ and $\rho_M$ 
descend to isomorphisms
\[
A\circ _A M\overset{\cong}{\longrightarrow} M\quad\text{and}\quad 
M\circ_B B\overset{\cong}{\longrightarrow} M,
\]
respectively. The same holds for 
$\lambda_N$ and $\rho_N$, of course.  

We say that a bimodule $1$-morphism is 
{\em biprojective} if it 
is projective as a left module $1$-morphism 
and as a right module $1$-morphism (but not 
necessarily as a bimodule $1$-morphism). 

\begin{theorem}\label{thm:Mor-Tak2}  
Let $\cC$ be a fiat $2$-category and $A$ and $B$ two algebra $1$-morphisms in 
$\overline{\cC}$. As $2$-representations of 
$\overline{\cC}$,  $\mathrm{mod}_{\overline{\ccC}}(A)$ and 
$\mathrm{mod}_{\overline{\ccC}}(B)$ are equivalent if and 
only if there exist biprojective bimodule $1$-morphisms    
$M={}_A M_B$ and $N={}_BN_A$ in $\overline{\cC}$ and 
bimodule $2$-isomorphisms 
\[
f\colon M\circ_B N\overset{\cong}{\longrightarrow} A 
\quad\text{and}\quad 
g\colon N\circ_A M\overset{\cong}{\longrightarrow} B, 
\]
such that the diagrams  

\[
\xymatrix@C=12mm{
 M\circ_B N\circ_A M    
 \ar[r]^/.25cm/{f\circ \mathrm{id}_M}\ar[d]_{\mathrm{id}_M\circ g} 
 &       A\circ_A M \ar[d]^{\lambda_M}\\
M\circ_B B    \ar[r]_/.25cm/{\rho_M} & M
}
\quad\quad
\xymatrix@C=12mm{
 N\circ_A M\circ_B N    
 \ar[r]^/.25cm/{g\circ \mathrm{id}_N}\ar[d]_{\mathrm{id}_N\circ f} 
 &        B\circ_B N \ar[d]^{\lambda_N}\\
N\circ_A A    \ar[r]_/.25cm/{\rho_N} & N
}
\]
commute. \qed
\end{theorem}

\begin{remark}\label{remark:graded2}
Again, all results above admit 
generalizations to the graded setup.
\end{remark}

An example of Morita(--Takeuchi) equivalent (co)algebra $1$-morphisms 
is given later on in Example \ref{example:MT-theory}.

\section{Constructing (co)algebra \texorpdfstring{$1$}{1}-morphisms using idempotents}\label{s7}

\subsection{Adjunctions and (co)monads}\label{s7.1}

Recall the following (see, e.g. \cite{ML}):

Let $(F,G)$ be a pair of adjoint 
functors $F\colon \mathcal{C}\to\mathcal{D}$ 
and $G\colon \mathcal{D}\to\mathcal{C}$,
for two categories $\mathcal{C}$ and $\mathcal{D}$. 
Let $\eta\colon \mathbbm{1}_{\mathcal{C}}\to GF$
and
$\epsilon\colon FG\to \mathbbm{1}_{\mathcal{D}}$ 
be the 
corresponding unit and counit of the adjunction.
Then the composition $GF$ carries the 
natural structure of a monad given by
$\eta$ and $\mathrm{id}_G\circ_0\epsilon\circ_0\mathrm{id}_F\colon GFGF\to GF$.
Further, the composition $FG$ carries the natural 
structure of a comonad given by
$\epsilon$ and $\mathrm{id}_F\circ_0\eta\circ_0\mathrm{id}_G\colon FG\to FGFG$.
In particular, $GF$ has the natural structure of
an algebra $1$-morphism in the $2$-category of endofunctors of $\mathcal{C}$, and $FG$ has the natural 
structure of a coalgebra $1$-morphism in the $2$-category of endofunctors of 
$\mathcal{D}$.

\subsection{(Co)algebra \texorpdfstring{$1$}{1}-morphisms for projective bimodules}\label{s7.2}

Let $A$ be a finite dimensional algebra 
and $A\text{-}\mathrm{mod}\text{-}A$ the 
bicategory of all $A$-$A$-bimodules (with respect to tensoring over $A$). 
Abusing notation, we will also denote by 
$A\text{-}\mathrm{mod}\text{-}A$ its strictification.
We also use the notation $A\text{-}\mathrm{mod}$ 
and $\mathrm{mod}\text{-}A$ in the evident way.

For an idempotent $e\in A$, consider the 
$A$-$\mathbbm{k}$-bimodule $Ae$ which 
defines an exact functor
\[
F=Ae\otimes_{\mathbbm{k}}{}_-\colon
\mathbbm{k}\text{-}\mathrm{mod}\to A\text{-}\mathrm{mod}.
\]
By the usual tensor-hom adjunction, the 
right adjoint of $F$ is the functor 
\[
G=\mathrm{Hom}_A(Ae,{}_-)\colon  A
\text{-}\mathrm{mod}\to \mathbbm{k}\text{-}\mathrm{mod}.
\]
The functor $G$ is exact and is isomorphic to the functor 
\[
G'=eA\otimes_A{}_-\colon  A
\text{-}\mathrm{mod}\to \mathbbm{k}\text{-}\mathrm{mod}.
\]
The morphisms corresponding to the counit and the unit of the pair $(F,G)$ are 
\[
\varepsilon\colon Ae\otimes_{\mathbbm{k}}eA\to A 
\]
given by multiplication and
\[
\eta\colon \mathbbm{k}\to eA\otimes_{A}Ae,
\quad 1\mapsto e\otimes e.
\]

The exact functor $F$ has a ``left exact'' 
representation via an isomorphic functor
\[
F'=\mathrm{Hom}_{\mathbbm{k}}(\mathrm{Hom}_{\mathbbm{k}}(Ae,\mathbbm{k}),{}_-)\colon
\mathbbm{k}\text{-}\mathrm{mod}\to A\text{-}\mathrm{mod}.
\] 
Therefore, $F$ has a left adjoint given 
by the functor
\[
H= \mathrm{Hom}_{\mathbbm{k}}(Ae,\mathbbm{k})\otimes_{A}{}_-\colon A\text{-}\mathrm{mod}\to \mathbbm{k}\text{-}\mathrm{mod}.
\]
The morphisms corresponding to the counit 
and the unit of the pair  $(H,F)$ are the evaluation morphism
\[
\varepsilon'\colon \mathrm{Hom}_{\mathbbm{k}}(Ae,\mathbbm{k})\otimes_{A}Ae\to \mathbbm{k},
\quad  
\beta\otimes a\mapsto \beta(a)  
\]
and the coevaluation morphism
\[
\eta'\colon A\to Ae\otimes_{\mathbbm{k}}\mathrm{Hom}_{\mathbbm{k}}(Ae,\mathbbm{k}),
\quad 1\mapsto {\textstyle\sum_i} a_i\otimes a_i^*,
\]
where $\{a_i\}$ is some fixed basis of 
$Ae$ and $\{a_i^*\}$ is the corresponding dual 
basis of 
the $\mathbbm{k}$-vector space $\mathrm{Hom}_{\mathbbm{k}}(Ae,\mathbbm{k})$.

To state the next proposition,
recall that an object $A$ which is both an algebra and a coalgebra object,
is called \textit{Frobenius} provided that comultiplication 
is a homomorphism of $A$-$A$-bimodules, see 
e.g. \cite[Definition 7.20.3]{EGNO} 
and also \cite{Mu,SF}. Henceforth, we use the 
notion of a \textit{Frobenius $1$-morphism}.

\begin{proposition}\label{prop31}
In the above setup we have: 
\begin{enumerate}[$($i$)$]
\item\label{prop31.1} The $A$-$A$-bimodule 
$Ae\otimes_{\mathbbm{k}}eA$ has the structure of
a coalgebra $1$-morphism.
\item\label{prop31.2} If $eA\cong \mathrm{Hom}_{\mathbbm{k}}(Ae,\mathbbm{k})$ 
in $\mathrm{mod}\text{-}A$, then
the $A$-$A$-bimodule $Ae\otimes_{\mathbbm{k}}eA$ has the structure 
of an algebra $1$-morphism.
\item\label{prop31.3} In the setup of \eqref{prop31.2}, 
the $A$-$A$-bimodule $Ae\otimes_{\mathbbm{k}}eA$ has the structure of 
a Frobenius $1$-morphism.
\end{enumerate}
\end{proposition}

\begin{proof}
Claim \eqref{prop31.1} follows 
from the comonad discussion in Section \ref{s7.1} 
applied to the 
pair $(F,G)$ of adjoint functors. Claim \eqref{prop31.2} 
follows from the monad discussion in
Section \ref{s7.1} applied to the pair $(H,F)$ 
of adjoint functors.

To prove claim \eqref{prop31.3}, we have to check 
the Frobenius condition, that is 
compatibility of the algebra and coalgebra structures.
This reduces to two commutative diagrams. 
For simplicity, we write $(Ae)^*$ for 
$\mathrm{Hom}_{\mathbbm{k}}(Ae,\mathbbm{k})$.
Then the first diagram is
\[
\xymatrix@C1.25cm{
(Ae)^*\otimes_A Ae\ar[d]_{\varepsilon'}
\ar[rr]^/-.6cm/{\eta\otimes\mathrm{id}_{(Ae)^*\otimes_A Ae}}&&
(Ae)^*\otimes_A Ae\otimes_{\Bbbk}(Ae)^*\otimes_A Ae
\ar[d]^{\mathrm{id}_{(Ae)^*\otimes_A Ae}\otimes \varepsilon'}\\
\Bbbk\ar[rr]_/-.6cm/{\eta}&&(Ae)^*\otimes_A Ae
}
\]
and its commutativity is checked, using definitions, by the computation
\[
\xymatrix{
\beta\otimes a\ar[rr]\ar[d] && \Phi(e)\otimes e \otimes \beta\otimes a\ar[d]\\
\beta(a)\ar[rr]&&\beta(a)\Phi(e)\otimes e,
} 
\]
where $\Phi\colon eA\overset{\cong}{\longrightarrow} (Ae)^*$ is a fixed isomorphism.

The second diagram is
\[
\xymatrix@C1.25cm{
(Ae)^*\otimes_A Ae\ar[d]_{\varepsilon'}
\ar[rr]^/-.6cm/{\mathrm{id}_{(Ae)^*\otimes_A Ae}\otimes\eta}&&
(Ae)^*\otimes_A Ae\otimes_{\Bbbk}(Ae)^*\otimes_A Ae
\ar[d]^{\varepsilon'\otimes \mathrm{id}_{(Ae)^*\otimes_A Ae}}\\
\Bbbk\ar[rr]_/-.6cm/{\eta}&&(Ae)^*\otimes_A Ae
}
\]
and its commutativity is checked, using definitions, by the computation
\[
\xymatrix{
\beta\otimes a\ar[rr]\ar[d] && \beta\otimes a\ar[d]\otimes\Phi(e)\otimes e  \\
\beta(a)\ar[rr]&&\beta(a)\Phi(e)\otimes e.
}
\]
This completes the proof of the proposition.
\end{proof}

\subsection{Duflo involutions as (co)algebra \texorpdfstring{$1$}{1}-morphism}\label{s7.3}

Let $\cC$ be a fiat $2$-category and $\mathcal{J}$ a two-sided cell in $\cC$.
Assume that $\cC$ is 
$\mathcal{J}$-simple and that $\mathcal{J}$ is 
strongly regular. Let $\mathcal{L}$ be a left cell 
in $\mathcal{J}$ and $\mathrm{G}$ be the
\textit{Duflo involution} in $\mathcal{L}$. (For these notions, 
see \cite[Proposition 17]{MM1}, \cite[Propositions 27 and 28]{MM6} and Section \ref{s2.1}.)

\begin{theorem}\label{thm32}
In the above setting, we have:
\begin{enumerate}[$($i$)$]
\item\label{thm32.1} The Duflo involution $\mathrm{G}$ is an algebra 
$1$-morphism in $\cC$. The corresponding $2$-rep\-re\-sen\-ta\-ti\-on 
$\mathrm{proj}_{\overline{\ccC}}(\mathrm{G})$
of $\cC$ is equivalent to the cell 
$2$-rep\-re\-sen\-ta\-ti\-on $\mathbf{C}_{\mathcal{L}}$.
\item\label{thm32.2} The Duflo involution $\mathrm{G}$  is a coalgebra 
$1$-morphism in $\cC$. The corresponding $2$-rep\-re\-sen\-ta\-ti\-on 
$\mathrm{inj}_{\underline{\ccC}}(\mathrm{G})$
of $\cC$ is equivalent to the cell $2$-rep\-re\-sen\-ta\-ti\-on 
$\mathbf{C}_{\mathcal{L}}$.
\item\label{thm32.3} The Duflo involution $\mathrm{G}$ is a Frobenius 
$1$-morphism in $\cC$. 
\end{enumerate}
\end{theorem}

\begin{proof}
We prove claim \eqref{thm32.1} and 
claim \eqref{thm32.2} will follow by duality.

Since the only relevant $1$-morphisms are those contained in 
$\mathcal{J}$, we may assume without loss of generality that $\mathcal{J}$ 
contains all indecomposable $1$-morphisms of $\cC$ that are not isomorphic to
some identity $1$-morphisms. Therefore, the classification of 
$\mathcal{J}$-simple $2$-categories in \cite[Theorem 13]{MM3} 
reduces our statement to the special case $\cC=\cC_A$, for some
weakly-symmetric finite dimensional basic algebra $A$.
($\cC_A$ is the $2$-category 
whose objects are $A_i\text{-}\mathrm{mod}$, 
with $A_i$ being the 
connected components of $A$, and whose 
morphism categories are generated by 
functors isomorphic to tensoring 
with projective $A_i$-$A_j$-bimodules. 
See e.g. \cite[Section 7.3]{MM1} for details on $\cC_A$.)
In the case $\cC=\cC_A$, the Duflo involution $\mathrm{G}$
has the form $Ae\otimes_{\mathbbm{k}}eA$, for some 
primitive idempotent $e\in A$, see e.g. \cite[Proposition 28]{MM6}. 
Therefore, existence of the algebra 
$1$-morphism structure on $\mathrm{G}$ follows immediately from 
Proposition \ref{prop31}\eqref{prop31.2}. 

It remains to prove the claim about $2$-representations.
Let $e=e_1,e_2,\dots,e_n$ be a complete list of 
pairwise orthogonal primitive idempotents of $A$. Then 
\[
\{Ae_i\otimes eA\,:\,i=1,2,\dots,n\} 
\]
is the list of all $1$-morphisms in $\mathcal{L}$, up to isomorphism.
By \cite[Section 7.3]{MM1}, the corresponding cell $2$-representation $\mathbf{C}_{\mathcal{L}}$
is equivalent to the defining action of $\cC_A$ on 
the category $A\text{-}\mathrm{proj}$ of projective $A$ modules.

Since $A$ is a self-injective and weakly symmetric algebra, there is an isomorphism 
$eA\cong \mathrm{Hom}_{\mathbbm{k}}(Ae,\mathbbm{k})$ of right $A$-modules
which we may use to identify
$eA$ and $\mathrm{Hom}_{\mathbbm{k}}(Ae,\mathbbm{k})$. Hence, we may express the unit morphism
\[
\eta\colon A\to Ae\otimes_{\mathbbm{k}}\mathrm{Hom}_{\mathbbm{k}}(Ae,\mathbbm{k})
\]
in the form $\eta(1)=\sum_i a_i\otimes a^*_i$, where $\{a_i\}$ is some basis of 
$Ae$ and $\{a_i^*\}$ is the corresponding dual 
basis of $\mathrm{Hom}_{\mathbbm{k}}(Ae,\mathbbm{k})$. 
Furthermore, the multiplication map 
\[
\big(Ae\otimes_{\mathbbm{k}}\mathrm{Hom}_{\mathbbm{k}}(Ae,\mathbbm{k})\big)\otimes_A
\big(Ae\otimes_{\mathbbm{k}}\mathrm{Hom}_{\mathbbm{k}}(Ae,\mathbbm{k})\big)\to
Ae\otimes_{\mathbbm{k}}\mathrm{Hom}_{\mathbbm{k}}(Ae,\mathbbm{k})
\]
is just given by contraction, i.e. $\big(a\otimes \varphi\big)\otimes \big(b\otimes\psi\big)\mapsto 
\varphi(b) \big(a\otimes \psi\big)$.

Each $Ae_i\otimes_{\mathbbm{k}}\mathrm{Hom}_{\mathbbm{k}}(Ae,\mathbbm{k})$, $i=1,2,\dots,n$, is naturally  
a right $\mathrm{G}$-module, and the additive 
closure of these objects is clearly 
stable under the left $\cC_A$-action. Let us denote 
the resulting $2$-representation of $\cC_A$
by $\mathbf{M}$. Further, the fact that homomorphisms in $\mathrm{mod}_{\ccC_A}(\mathrm{G})$
must commute with the $\mathrm{G}$-action 
implies that
\[
\mathrm{Hom}_{\mathrm{mod}_{\ccC_A}(\mathrm{G})} 
\big(Ae_i\otimes_{\mathbbm{k}}\mathrm{Hom}_{\mathbbm{k}}(Ae,\mathbbm{k}),
Ae_j\otimes_{\mathbbm{k}}\mathrm{Hom}_{\mathbbm{k}}(Ae,\mathbbm{k})\big)
\]
is the $e_iAe_j\otimes_{\mathbbm{k}} e$ subalgebra of 
\[
\mathrm{Hom}_{\ccC_A}  \big(Ae_i\otimes_{\mathbbm{k}}\mathrm{Hom}_{\mathbbm{k}}(Ae,\mathbbm{k}),
Ae_j\otimes_{\mathbbm{k}}\mathrm{Hom}_{\mathbbm{k}}(Ae,\mathbbm{k})\big)\cong e_iAe_j\otimes_{\mathbbm{k}}eAe.
\]

This means that the Cartan matrix of the underlying 
algebra of $\mathbf{M}$ coincides with the
Cartan matrix of $A$. Thus, \cite[Theorem 4]{MM6} implies 
that $\mathbf{M}$ is equivalent 
to $\mathbf{C}_{\mathcal{L}}$.

On the other hand, any $X\in\mathrm{mod}_{\ccC_A}(\mathrm{G})$ is a quotient of 
$XAe\otimes_{\mathbbm{k}}\mathrm{Hom}_{\mathbbm{k}}(Ae,\mathbbm{k})$, by 
the argument dual to the one
used in the last paragraph of the proof of Theorem \ref{thm7}. This implies that $\mathbf{M}$
coincides with the $2$-representation $\mathrm{proj}_{\overline{\ccC}}(\mathrm{G})$
of $\cC$ and completes the proof of claim \eqref{thm32.1}.

Claim \eqref{thm32.3} follows from the above and 
Proposition \ref{prop31}\eqref{prop31.3}.
\end{proof}

\begin{example}\label{example:MT-theory}
For the class of (co)algebra $1$-morphisms in Proposition \ref{prop31}, it is not hard 
to work out the Morita(--Takeuchi) $2$-theory from 
Section \ref{section:Morita-Takeuchi} 
explicitly. 

Suppose that $e$ and $f$ are two non-zero primitive idempotents in 
a finite dimensional self-injective algebra $A$. 
It follows from the 
proof of Theorem \ref{thm32} that the corresponding coalgebra $1$-morphisms 
\[
E_e=Ae\otimes_{\mathbb{k}} eA\quad\text{and}\quad E_f=Af\otimes_{\mathbb{k}}fA
\]
are Morita--Takeuchi equivalent. 
More explicitly, the equivalence and its inverse are given by 
\begin{gather*}
\begin{aligned}
{}_-\; \square_{E_e} \left(Ae \otimes_{\mathbb{k}} fA\right) &\colon \mathrm{comod}_{A\text{-}\mathrm{mod}\text{-}A}(E_e)
\overset{\cong}{\longrightarrow} \mathrm{comod}_{A\text{-}\mathrm{mod}\text{-}A}(E_f) \\
{}_-\; \square_{E_f} \left(Af \otimes_{\mathbb{k}} eA\right) &\colon \mathrm{comod}_{A\text{-}\mathrm{mod}\text{-}A}(E_f)
\overset{\cong}{\longrightarrow} \mathrm{comod}_{A\text{-}\mathrm{mod}\text{-}A}(E_e).
\end{aligned}
\end{gather*}
Moreover, $E_e$  and $E_f$ are also Morita equivalent algebra 
$1$-morphisms. This time, the equivalence and its inverse are given by 
\begin{gather*}
\begin{aligned}
{}_- \circ_{E_e} \left(Ae \otimes_{\mathbb{k}} fA\right) &\colon \mathrm{mod}_{A\text{-}\mathrm{mod}\text{-}A}(E_e)
\overset{\cong}{\longrightarrow} \mathrm{mod}_{A\text{-}\mathrm{mod}\text{-}A}(E_f) \\
{}_- \circ_{E_f} \left(Af \otimes_{\mathbb{k}} eA\right) &\colon \mathrm{mod}_{A\text{-}\mathrm{mod}\text{-}A}(E_f)
\overset{\cong}{\longrightarrow} \mathrm{mod}_{A\text{-}\mathrm{mod}\text{-}A}(E_e).
\end{aligned}
\end{gather*}
\end{example}

\subsection{Wall-crossings as Frobenius \texorpdfstring{$1$}{1}-morphism}\label{s7.4}

The constructions described in the previous subsections admit a
generalization as follows. 

Let $A$ be a finite dimensional unital 
$\mathbbm{k}$-algebra
and $B$ be a unital subalgebra of $A$. Assume that 
$A$ is projective, both as a left
and as a right $B$-module, and that both algebras, $A$ 
and $B$, are symmetric in the sense that 
there exist bimodule isomorphisms 
${}_AA_A\cong {}_A\mathrm{Hom}_{\mathbbm{k}}(A,\mathbbm{k})_A$
and ${}_BB_B\cong {}_B\mathrm{Hom}_{\mathbbm{k}}(B,\mathbbm{k})_B$. 
Then the induction and 
restriction functors 
\[
{}_AA\otimes_B{}_-\colon B\text{-}\mathrm{mod}\to  A\text{-}\mathrm{mod}\quad\text{ and }\quad
{}_BA\otimes_A{}_-\colon A\text{-}\mathrm{mod}\to  B\text{-}\mathrm{mod}
\]
are biadjoint, cf. \cite[Section 6.4]{MM6}. Consequently, 
similarly to Proposition \ref{prop31}, the $A$-$A$-bimodule 
${}_AA\otimes_B A_A$ has the 
structure of a Frobenius $1$-morphism.

The above can be applied to the following 
situation: Let $(W,S)$ be a finite Coxeter
system with a fixed reflection representation 
$\mathfrak{h}$ of $W$ and  
$\mathtt{C}$ the corresponding coinvariant algebra. 
Let $\cS=\cS(W,S,\mathfrak{h})$
denote the associated $2$-category of {\em Soergel $\mathtt{C}$-$\mathtt{C}$-bimodules},
see \cite{So1,So2,EW}. Then the indecomposable Soergel 
bimodules are naturally indexed
by elements in $W$ and, for $w\in W$, we denote by $B_w$ 
the corresponding 
indecomposable Soergel bimodule.

For $X\subset S$, let $W_X$ be the 
corresponding parabolic 
subgroup of $W$. Denote by $w_0^X$ the 
longest element in $W_X$.

\begin{proposition}\label{prop37}
The $\mathtt{C}$-$\mathtt{C}$-bimodule 
$B_{w_0^X}$ has the natural  
structure of a Frobenius $1$-morphism in $\cS$.
\end{proposition}

\begin{proof}
Let $\mathtt{C}^X$ denote the algebra 
of $X$-invariants in $\mathtt{C}$. Then
both, $\mathtt{C}$ and $\mathtt{C}^X$, are 
symmetric algebras and $\mathtt{C}$
is projective as a left and as a right $\mathtt{C}^X$-module, 
see e.g. \cite{Hi}. Moreover,
the $\mathtt{C}$-$\mathtt{C}$-bimodule $B_{w_0^X}$ is isomorphic to
$\mathtt{C}\otimes_{\mathtt{C}^X}\mathtt{C}$, 
see e.g. \cite[Section 3.4]{So1} which also
admits a straightforward 
generalization to finite Coxeter groups. 
The claim follows from the discussion in Section \ref{s7.2}.
\end{proof}

\begin{remark}\label{remark:graded3}
As in Remarks \ref{remark:graded} and \ref{remark:graded2}, 
the results above generalizes without 
difficulties to the graded world.
\end{remark}

\section{Application to Soergel bimodules for dihedral groups}\label{s8}

In this section we work over $\mathbbm{k}=\mathbb{C}$. 
Moreover, we fix a positive integer $n>2$.

\subsection{Various \texorpdfstring{$2$}{2}-categories: from affine \texorpdfstring{$\mathfrak{sl}_2$}{sl2} to singular Soergel bimodules}\label{s8.1}

We denote by $\cA_n$ the category of representations of the affine Lie algebra 
$\widehat{\mathfrak{sl}_2}$ at level $n-2$, see \cite{Ka} and 
also \cite[Section 6]{Os}
and references therein for details. The category $\cA_n$ has 
the structure of a (non-strict) monoidal category via the so-called
{\em fusion product}, see e.g. \cite[Section 2.11]{Fi}. 
We can therefore consider $\cA_n$ as a bicategory.
Abusing notation, we will also denote by 
$\cA_n$ the strictification of the latter bicategory.

For a primitive complex $2n$-th root of unity $q$, consider 
also the semisimple subquotient
$\cQ_n$ of the category of integrable representations of 
the quantum group 
$\mathrm{U}_q(\mathfrak{sl}_2)$, see e.g. \cite{GK,An,AP}. 
The category $\cQ_n$ can be seen as 
the additive closure of simple finite dimensional (highest weight) 
modules $L_k$ of quantum dimension $[k+1]_q$, where $k=0,1,\dots,n-2$. 
Again, $\cQ_n$ has the structure of a (non-strict) monoidal 
category given by the quantum {\em fusion product}. 
Abusing notation, we will also denote by 
$\cQ_n$ the strictification of the latter bicategory. 

According to \cite{Fi}, the 
$2$-categories $\cQ_n$ and $\cA_n$ are 
equivalent. This equivalence is based on 
the results of \cite{KL1,KL2,KL3,KL4,Lu}.

Moreover, $\cQ_n$ also has a 
diagrammatic presentation, because it is equivalent to the Karoubi envelope 
of the one-object $2$-category of all Temperley-Lieb diagrams 
modulo the $2$-ideal of the negligible ones (see e.g. \cite[Section XII.7]{Tu}). 
Elias \cite[Section 4]{El} uses a closely related 
$2$-category, denoted $\widetilde{\cT\cL_n}(\delta)$ 
and called the \textit{two-color Temperley-Lieb $2$-category}, 
with the two objects (``colors'') $s$ and $t$. 
In $\cT\cL_n$, the regions of the Temperley-Lieb 
diagrams are colored by $s$ or $t$, such that any 
two regions separated by one strand have different colors. This 
ensures that the coloring of any diagram is 
uniquely determined by the color of its rightmost region. Vice versa, 
we can extend any color of the rightmost 
region of given diagram uniquely to a coloring of the whole diagram. 
Finally, let $\cT\cL_n$ denote the quotient 
of $\widetilde{\cT\cL_n}(q+q^{-1})$ by 
the $2$-ideal generated
by the negligible Jones-Wenzl projector corresponding to our 
choice of $n$, see
\cite[Section 4]{El} for details. (Here $q$ is the same 
root of unity as before.)

Consider also the bicategory 
$\widetilde{\cS\cS}_n$ of {\em singular Soergel bimodules} 
(over the polynomial algebra) for the dihedral 
group
\[
D_{2n}=\langle s,t\mid s^2=t^2=1, \underbrace{sts\cdots}_{n\text{ factors}}=\underbrace{tst\cdots}_{n\text{ factors}}\rangle
\]
of order $2n$, and denote by ${\cS\cS}_n$ 
the quotient of $\widetilde{\cS\cS}_n$ by the $2$-ideal 
generated by all $D_{2n}$-invariant 
polynomials of positive degree. Then ${\cS\cS}_n$ is the bicategory of singular 
Soergel bimodules over 
the coinvariant algebra $\mathtt{C}$ of $D_{2n}$. 
We refer to \cite[Section 6.1]{El} 
and \cite{Wi} for details. (We also use a notation similar to that in Section \ref{s7.4}.)
We will again denote the corresponding 
strictifications by the same symbols.

We denote by $\widehat{{\cS\cS}_n}$ 
the {\em small quotient} of
${\cS\cS}_n$, that is the 
quotient of ${\cS\cS}_n$ by the $2$-ideal 
generated by the
indecomposable Soergel bimodule corresponding to the longest 
element in $D_{2n}$.
Further, we denote by $\cS_n$ the endomorphism $2$-category of the regular 
object in ${\cS\cS}_n$,
that is the $2$-category of (regular or usual) Soergel 
bimodules and similarly for 
$\widehat{{\cS}_n}$. Again, we denote their corresponding 
strictifications by the same symbols.

\begin{remark}\label{remark:same-2-reps}
It follows from Corollary \ref{cor11} that there is a bijection 
between simple transitive and faithful $2$-representations 
of $\widehat{{\cS}_n}$ and
$\widehat{{\cS\cS}_n}$. Note that Corollary \ref{cor11} is 
applicable due to the fact that 
the identity $1$-morphism on the $s$-singular object of 
$\widehat{{\cS\cS}_n}$ factors through 
$\widehat{{\cS}_n}$, which implies that every faithful 
$2$-representation of $\widehat{{\cS\cS}_n}$
must be supported, in particular, on this $s$-singular object. 
(Similarly for $t$.)
\end{remark}

\subsection{Based modules and algebra \texorpdfstring{$1$}{1}-morphisms for affine \texorpdfstring{$\mathfrak{sl}_2$}{sl2}}\label{s8.2}

The papers \cite{DZ} and \cite{EK} study 
and classify so-called indecomposable \textit{based} 
or \textit{$\mathbb{Z}_{+}$ modules}
over the split Grothendieck group 
$[\cA_n]_{\oplus}$ of $\cA_n$
(also called the {\em Verlinde algebra}). 
Such modules 
turn out to be in one-to-one correspondence 
with finite Dynkin diagrams 
and so-called tadpole diagrams which have Coxeter 
number $n$.

\begin{theorem}\label{thm57}(\cite[Theorem 6.1]{Os}, see also \cite{BEK}.)
For each simply laced Dynkin diagram $\Gamma$ with Coxeter number $n$, there is a unique,
up to isomorphism, algebra $1$-morphism $A_{\Gamma}$ in $\cA_{n}$ such that the based 
$[\cA_n]_{\oplus}$-module $[\mathrm{proj}_{\ccA_{n}}(A_{\Gamma})]_{\oplus}$ corresponds
to $\Gamma$ via the equivalence in \cite{EK}.
\end{theorem}

We stress
that for the non-simply laced finite 
Dynkin diagrams and the tadpole diagrams appearing 
in \cite[Section 3.3]{EK} the corresponding algebra 
$1$-morphisms do not exist.

\subsection{Algebra \texorpdfstring{$1$}{1}-morphisms for Soergel bimodules}\label{s8.3}

Certain 
indecomposable 
ba\-sed modules over
the split Grothendieck group $[\widehat{{\cS}_n}]_{\oplus}$ are 
in one-to-one correspondence 
with Dynkin diagrams of ADE type and Coxeter number $n$, 
as it was shown in \cite[Sections 6 and 7]{KMMZ}.
There are corresponding algebra $1$-morphisms:

\begin{theorem}\label{thm58}
For each simply laced Dynkin diagram $\Gamma$ with Coxeter number $n$, there is an 
algebra $1$-morphism $B_{\Gamma}$ in $\widehat{{\cS}_{n}}$ 
such that the based 
$[\widehat{{\cS}_n}]_{\oplus}$-module 
$[\mathrm{proj}_{\widehat{{\ccS}_{n}}}(B_{\Gamma})]_{\oplus}$ 
corresponds to $\Gamma$ via the bijection 
in \cite[Sections 6 and 7]{KMMZ}.
\end{theorem}

\begin{proof} 
First consider the algebra $1$-morphism $B_{\Gamma}^s$ given by 
the image of $A_{\Gamma}$ from Theorem \ref{thm57} via 
the composition of the following $2$-functors:

\begin{itemize}
\item Finkelberg's equivalence between the two $2$-categories $\cA_{n}$ and $\cQ_n$;

\item the fully-faithful embedding of the $2$-category
$\cQ_n$ into $\cT\cL_n$ given by 
\[
\underbrace{L_1\otimes\dots\otimes L_1}_{k}
\mapsto
\underbrace{\dots\hspace*{-.21cm}\phantom{L_1}\hspace*{-.21cm} ststs}_{k},
\quad k=0,1,\dots,n-2,
\]
with the usual assignments on morphisms, where $L_1$ denotes the vector representation 
of $\mathrm{U}_q(\mathfrak{sl}_2)$ (which monoidally generates
$\cQ_n$);

\item the faithful $2$-functor from the $2$-category $\cT\cL_n$ to $\widetilde{{\cS\cS}_n}$ 
defined by Elias in \cite[Proposition 1.2 and Theorem 6.29]{El}, which is also 
full onto degree-zero $2$-morphisms, composed with the projection onto $\widehat{{\cS\cS}_n}$.

\end{itemize}

Notice that the second functor is the 
composite of the equivalence between $\cQ_n$ and the quotient of 
the Temperley-Lieb $2$-category, mentioned in Section \ref{s8.1}, and the fully-faithful 
embedding of the latter into $\cT\cL_n$ defined by coloring 
the right-most region of every diagram by $s$. (Alternatively, one could 
color it by $t$, of course.)

Let us check that the indecomposable based 
$[\widehat{{\cS\cS}_n}]_{\oplus}$-module 
$[\mathrm{proj}_{\widehat{{\ccS}_{n}}}(B_{\Gamma}^s)]_{\oplus}$ 
has the correct combinatorics. 

To this end, we recall that all involved $2$-categories 
have a natural positive grading
and all involved $2$-functors are gradable. 
(Note that our whole setup is applicable so far, 
cf. Remarks \ref{remark:graded}, \ref{remark:graded2} 
and \ref{remark:graded3}.) The correctness 
of the combinatorics in question follows if we can show that the (graded version of) the above 
composition is full on $2$-morphisms of degree zero.
Indeed, if the pushforward $Y$ of an indecomposable projective $X$ in 
$\mathrm{proj}_{\ccA_{n}}(A_{\Gamma})$ were 
decomposable, that would mean existence of 
a non-trivial idempotent in the endomorphism ring of $Y$, 
considered as a $1$-morphism in 
$\mathrm{proj}_{\widehat{{\ccS}_{n}}}(B_{\Gamma}^s)$. Next, from fullness on degree zero $2$-morphism,
we would get a non-trivial idempotent in the endomorphism ring 
of $X$, considered as a
$1$-morphism of $\cA_{n}$. However, due to faithfulness of the above 
composition, this idempotent will also
live in $\mathrm{proj}_{\ccA_{n}}(A_{\Gamma})$, a contradiction.

Hence, it remains to check fullness of the three $2$-functors 
from above on $2$-morphisms of degree zero.
For the first $2$-functor of the above composition the claim 
about fullness on $2$-morphisms of degree zero is clear. For 
the last $2$-functor of the composition
this is contained in \cite[Proposition 1.2]{El}. For the 
second $2$-functor, such a claim follows
directly from the definitions.

By construction and the properties of Elias'
$2$-functor from \cite[Proposition 1.2]{El}, $B_{\Gamma}^s$ lives in the 
endomorphism $2$-category of 
the $s$-singular object in $\widehat{{\cS}\cS_{n}}$.
Let  
\[
B_{\Gamma}= \mathtt{C}\otimes_{\mathtt{C}^s}  B_{\Gamma}^s \otimes_{\mathtt{C}^s}\mathtt{C},
\]
where, as above, $\mathtt{C}$ is the 
coinvariant algebra of $D_{2n}$ and $\mathtt{C}^s$ 
the subalgebra which is the quotient 
of the subalgebra of $s$-invariant polynomials. Note that 
\[
B_{\Gamma}\otimes_\mathtt{C}B_{\Gamma} \cong  \mathtt{C}\otimes_{\mathtt{C}^s}  B_{\Gamma}^s \otimes_{\mathtt{C}^s}\mathtt{C} \otimes_{\mathtt{C}^s}  B_{\Gamma}^s \otimes_{\mathtt{C}^s}\mathtt{C}
\]
maps to 
\[
\mathtt{C}\otimes_{\mathtt{C}^s}  B_{\Gamma}^s \otimes_{\mathtt{C}^s}  B_{\Gamma}^s \otimes_{\mathtt{C}^s}\mathtt{C},
\]
by applying the adjunction morphism given by the Demazure operator
$\partial_s\colon \mathtt{C}\to \mathtt{C}^s$ 
(see \cite[Section 3.6]{El} 
for the definition) to the tensor factor in the middle (we are omitting gradings, for 
simplicity). Composing this map with the multiplication morphism of $B_{\Gamma}^s$ gives the multiplicative structure on 
$B_{\Gamma}$. 

The unital structure $\mathtt{C}\to  B_{\Gamma}$ is obtained from the unital structure 
$\mathtt{C}^s\to B_{\Gamma}^s$, by 
tensoring on both sides with $\mathtt{C}$ over $\mathtt{C}^s$ and precomposing with usual adjunction 
bimodule map $\mathtt{C}\to \mathtt{C}\otimes_{\mathtt{C}^s}\mathtt{C}$ 
given by $1\mapsto
{}^{1}\! / \!{}_{2}(\alpha_s \otimes 1 + 1\otimes \alpha_s)$, where $\alpha_s$ is the simple root 
corresponding to the reflection $s$ (see e.g. \cite[Example 3.4]{El}). 

To prove associativity of multiplication for $B_{\Gamma}$, consider the following diagram:
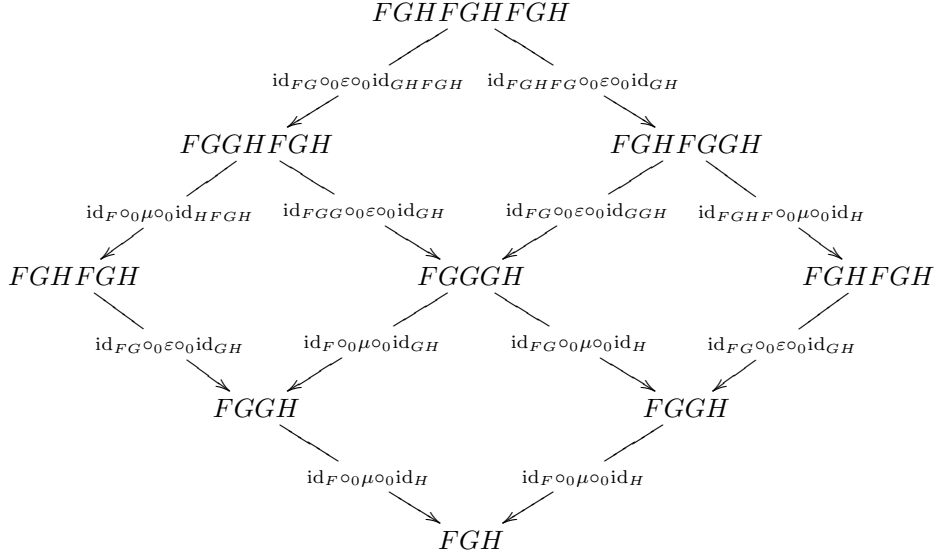
\[
\xymatrix@R1.3cm@C.3cm{
&&FGHFGHFGH\ar[dr]|{\somespace\mathrm{id}_{FGHFG}\circ_0\varepsilon\circ_0\mathrm{id}_{GH}}
\ar[dl]|{\somespace\mathrm{id}_{FG}\circ_0\varepsilon\circ_0\mathrm{id}_{GHFGH}}&&\\
&FGGHFGH\ar[dl]|{\somespace\mathrm{id}_{F}\circ_0\mu\circ_0\mathrm{id}_{HFGH}}
\ar[dr]|-{\somespace\mathrm{id}_{FGG}\circ_0\varepsilon\circ_0\mathrm{id}_{GH}}&&
FGHFGGH\ar[dl]|-{\somespace\mathrm{id}_{FG}\circ_0\varepsilon\circ_0\mathrm{id}_{GGH}}
\ar[dr]|{\somespace\mathrm{id}_{FGHF}\circ_0\mu\circ_0\mathrm{id}_{H}}&\\
FGHFGH\ar[dr]|{\somespace\mathrm{id}_{FG}\circ_0\varepsilon\circ_0\mathrm{id}_{GH}}
&&FGGGH\ar[dl]|-{\somespace\mathrm{id}_{F}\circ_0\mu\circ_0\mathrm{id}_{GH}}
\ar[dr]|-{\somespace\mathrm{id}_{FG}\circ_0\mu\circ_0\mathrm{id}_{H}}
&&FGHFGH\ar[dl]|{\somespace\mathrm{id}_{FG}\circ_0\varepsilon\circ_0\mathrm{id}_{GH}}\\
&FGGH\ar[dr]|{\somespace\mathrm{id}_{F}\circ_0\mu\circ_0\mathrm{id}_{H}}
&&FGGH\ar[dl]|{\somespace\mathrm{id}_{F}\circ_0\mu\circ_0\mathrm{id}_{H}}&\\
&&FGH&&
}
\]
Here $G$ stands for the algebra $1$-morphism
$B_{\Gamma}^s$ and $\mu\colon GG\to G$ for the corresponding multiplication $2$-morphism. Further,
$F$ and $H$ denote translations out of the $s$-wall (i.e. 
${}_{\mathtt{C}}\mathtt{C}_{\mathtt{C}^s} \otimes_{\mathtt{C}^s} {}_-$) and to the 
$s$-wall (i.e. ${}_{\mathtt{C}^s}\mathtt{C}_{\mathtt{C}} \otimes_{\mathtt{C}}{}_-$), 
respectively, with $\varepsilon\colon HF\to\mathrm{Id}$ being the counit of the adjunction (given by $\partial_s$). The top rhombus and the
two rhombi on the sides commute by the interchange law. The bottom rhombus commutes due to 
associativity of $\mu$. Therefore the whole diagram commutes which yields associativity of 
multiplication for $B_{\Gamma}$. 

To prove unitality of $B_{\Gamma}$ consider the diagram
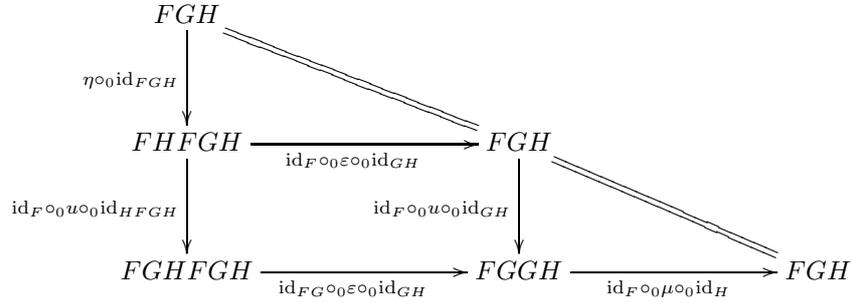
\[
\xymatrix@R1.25cm@C1.25cm{ 
FGH\ar@{=}[rrd]\ar[d]_{\eta\circ_0\mathrm{id}_{FGH}}&&\\
FHFGH\ar[d]_{\mathrm{id}_{F}\circ_0 u\circ_0\mathrm{id}_{HFGH}}
\ar[rr]_{\mathrm{id}_{F}\circ_0 \varepsilon\circ_0\mathrm{id}_{GH}}
&&FGH\ar@{=}[rrd]\ar[d]_{\mathrm{id}_{F}\circ_0 u\circ_0\mathrm{id}_{GH}}&\\
FGHFGH\ar[rr]_{\mathrm{id}_{FG}\circ_0 \varepsilon\circ_0\mathrm{id}_{GH}}&&
FGGH\ar[rr]_{\mathrm{id}_{F}\circ_0 \mu\circ_0\mathrm{id}_{H}}&&FGH
}
\]
with the same notation as above and, additionally, 
where $u\colon\mathrm{Id}\to G$ denotes the unit morphism for $G$ and
$\eta\colon\mathrm{Id}\to FH$ is the unit of the adjunction.
Here the top triangle commutes by adjunction, the square commutes by the interchange law and
the right triangle commutes due to unitality of $B_{\Gamma}^s$. Therefore the whole
diagram commutes and proves unitality of $B_{\Gamma}$. Hence, $B_{\Gamma}$ is indeed an
algebra $1$-morphism.

Since we have $\mathtt{C}\cong\mathtt{C}^s\{1\} \oplus\mathtt{C}^s\{-1\}$ as $\mathtt{C}^s$-$\mathtt{C}^s$-bimodules, the restriction of $B_{\Gamma}$ to $\mathtt{C}^s$ (on both sides) is isomorphic 
to $B_{\Gamma}^s\{2\} \oplus B_{\Gamma}^s\oplus B_{\Gamma}^s\oplus B_{\Gamma}^s\{-2\}$
as an algebra $1$-morphism. This 
implies that the simple transitive $2$-representation of $\widehat{{\cS}_{n}}$ corresponding to 
the algebra $1$-morphism $B_{\Gamma}$ can 
be obtained from the one of $\widehat{{\cS\cS}_n}$ corresponding 
to $B_{\Gamma}^s$ by restriction as in Corollary \ref{cor11}. 
\end{proof}

Thanks to Theorem \ref{thm7} and the 
construction in \cite[Section 3.6]{MM6}, there are always
many non-equivalent $2$-representations of $\widehat{{\cS}_{n}}$ 
whose Grothendieck groups 
give rise to the same based $[\widehat{{\cS}_n}]_{\oplus}$-module. 
(These are obtained by an inflation process from the 
ones constructed via Theorem \ref{thm58}.) 
Therefore one cannot expect any
direct uniqueness statement in Theorem \ref{thm58} similar to the 
one of Theorem \ref{thm57}.
However, we expect that under some additional assumptions of simplicity
together with specification of $\mathtt{t}\in\cC$ identifying the
endomorphism $2$-category in which the algebra $1$-morphism lives, the 
algebra $1$-morphism in the formulation 
of Theorem \ref{thm58} should be unique (in which case it, most 
probably, will not be the one 
constructed in the proof but rather a ``simple quotient'' of the latter). 
See also \cite[Theorem II]{MT}.

\subsection{Some concluding remarks}\label{remark:typeADE}

\begin{remark}\label{remark:typeADE-1}
If $\Gamma$ is of Dynkin type $A$, then 
the corresponding simple transitive $2$-representation of 
$\widehat{{\cS}_n}$ is equivalent to a cell 
$2$-representation (in the sense of Section \ref{s2.1}), 
as was shown in \cite[Sections 6 and 7]{KMMZ}. 
Furthermore, by the results in \cite[Table 1]{KO} and Theorem \ref{thm58}, 
we have $B_\Gamma=B_s$. The algebra structure on $B_s$ is given by the 
usual degree zero bimodule map $B_s \otimes_{\mathtt{C}} B_s\to B_s\{1\}$ 
of the Soergel calculus (the one which corresponds to the 
``merge'' in Elias' two-color Soergel calculus, cf. \cite[Section 5.3]{El}).

If $\Gamma$ is of Dynkin type $D$, then 
the simple transitive 
$2$-representation of $\widehat{{\cS}_n}$ 
can be constructed from a 
cell $2$-representation using the orbits under an 
involution. 
This was shown in \cite[Section 7]{KMMZ}. By the 
results in \cite[Table 1]{KO} and Theorem \ref{thm58}, we have 
$B_\Gamma=B_s\oplus B_{s w_0}$. (Here $w_0$ 
denotes the longest word 
in the dihedral group $D_{2n}$.)
The $2$-morphisms which define the algebra 
structure on $B_{\Gamma}$, can 
be deduced from the results in \cite[Section 7]{KO} and the 
relation between Temperley-Lieb diagrams and 
Elias' two-color Soergel calculus. It is not 
hard to check by hand that these induce an algebra 
structure on $B_{\Gamma}$. (See also Remark \ref{remark:typeADE-2}.)

If $\Gamma$ is of Dynkin type $E$, the situation 
is different. 
The existence of 
the corresponding simple transitive $2$-representation 
of $\widehat{{\cS}_n}$ was predicted 
in \cite[Section 7.5]{KMMZ} and its construction was given
in \cite{MT} using 
Elias' diagrammatic two-color Soergel calculus. 
From the results in \cite[Table 1]{KO} and 
Theorem \ref{thm58}, we obtain the decomposition of 
$B_{\Gamma}$ into indecomposable Soergel bimodules. However, the 
$2$-morphisms which define the algebra structure on $B_{\Gamma}$ are only 
determined up to scalars. Fixing these scalars is 
hard. In \cite[Theorem 6.1]{KO} 
and \cite[Theorem 6.1]{Os} it is shown that this is possible in a 
roundabout way, using arguments from conformal field theory. We were 
unable to prove the analogous result in $\widehat{{\cS}_n}$ directly 
and by hand due to the complexity of the diagrams involved. It would be 
interesting to have such a proof. 
\end{remark}

\begin{remark}\label{remark:typeADE-2}
We stress another conceptual 
difference between the $B_\Gamma$ for $\Gamma$ of Dynkin types $A$ and $D$ 
on one side, and those for $\Gamma$ of Dynkin type $E$ 
on the other: Using the description 
of the multiplication of Kazhdan--Lusztig basis 
elements in $[\widehat{{\cS}_n}]_{\oplus}$ (see e.g. \cite[Section 4]{dC}), 
it is not hard to see that the $B_{\Gamma}$ for types $A$ and $D$ 
descend to (pseudo-)idempotents in $[\widehat{{\cS}_n}]_{\oplus}$.
However, this is not true for the $B_{\Gamma}$ of type $E$.
\end{remark}



M.M.: Center for Mathematical Analysis, Geometry, and Dynamical Systems, Departamento de Matem{\'a}tica, 
Instituto Superior T{\'e}cnico, 1049-001 Lisboa, PORTUGAL \& Departamento de Matem{\'a}tica, FCT, 
Universidade do Algarve, Campus de Gambelas, 8005-139 Faro, PORTUGAL, email: {\tt mmackaay\symbol{64}ualg.pt}

Vo.Ma.: Department of Mathematics, Uppsala University, Box. 480,
SE-75106, Uppsala, SWEDEN, email: {\tt mazor\symbol{64}math.uu.se}

Va.Mi.: School of Mathematics, University of East Anglia,
Norwich NR4 7TJ, UK, email: {\tt v.miemietz\symbol{64}uea.ac.uk}

D.T.: Mathematisches Institut, Universit{\"a}t Bonn, Endenicher Allee 60, Room 1.003,
D-53115 Bonn, GERMANY, email: {\tt dtubben\symbol{64}math.uni-bonn.de}

\end{document}